\documentclass[a4paper]{amsart}

\usepackage[left=3.5cm, right=3.5cm, top=3.8cm, bottom=3.8cm]{geometry}
\usepackage[english]{babel}
\usepackage[utf8]{inputenc}
\usepackage{amsmath, amssymb, amsthm, mathtools, stmaryrd, enumitem, xcolor, bm, calc, todonotes}

\usepackage{manfnt}

\usepackage{mathrsfs} 
\definecolor{A}{HTML}{e9c604}
\definecolor{B}{HTML}{bace03}
\definecolor{C}{HTML}{a88d11}

\usepackage{tikz}
\usepackage{tikz-cd}
\usetikzlibrary{%
    matrix,%
    calc,%
    arrows,%
    positioning,
    cd
}

\definecolor{linkred}{rgb}{0.6,0,0}
\usepackage[
	pdftitle={Instantons and multibananas},
	pdfauthor={Oliver Leigh},
	ocgcolorlinks,
	linkcolor=black,
	citecolor=linkred,
	urlcolor=linkred]
{hyperref}

 \setlist[enumerate]{leftmargin=4em}
\newtheorem{thm}{Theorem}[section]
\newtheorem{cor}[thm]{Corollary}
\newtheorem{lem}[thm]{Lemma}
\newtheorem{prop}[thm]{Proposition}

\theoremstyle{definition}
\newtheorem{defn}[thm]{Definition}
\newtheorem{defnthm}[thm]{Definition-Theorem}
\newtheorem{rem}[thm]{Remark}

\numberwithin{equation}{section}

\makeatletter
\def\thm@space@setup{%
  \thm@preskip=1.5em plus 0.5em minus 0.5em
  \thm@postskip=\thm@preskip 
}
\makeatother

\newcommand{\multisum}[1]{\sum_{\scriptsize \begin{array}{c} #1 \end{array}}}
\newcommand{\multiprod}[1]{\prod_{\scriptsize \begin{array}{c} #1 \end{array}}}

\newcommand{\CC}{{\mathbb{C}}}
\newcommand{\EE}{{\mathbb{E}}}
\newcommand{\HH}{{\mathbb{H}}}
\newcommand{\LL}{{\mathbb{L}}}
\newcommand{\NN}{{\mathbb{N}}}
\newcommand{\PP}{{\mathbb{P}}}
\newcommand{\QQ}{{\mathbb{Q}}}
\newcommand{\TT}{{\mathbb{T}}}
\newcommand{\ZZ}{{\mathbb{Z}}}

\newcommand{\calE}{{\mathcal{E}}}

\newcommand{\calX}{{\mathcal{X}}}
\newcommand{\calZ}{{\mathcal{Z}}}

\newcommand{\Ind}{{\mathrm{Ind}}}
\newcommand{\Def}{{\mathrm{Def}}}
\newcommand{\Obs}{{\mathrm{Obs}}}

\newcommand{\psl}{[\hspace{-0.15em}[}
\newcommand{\psr}{]\hspace{-0.15em}]}

\newcommand{\vir}{{\mathrm{vir}}}

\renewcommand{\O}{{\mathcal{O}}}

\newcommand{\vi}[1]{{\bm{\mathsf{z}}_{\bm{i}#1}}}
\newcommand{\ui}[1]{{\bm{\mathsf{x}}_{\bm{i}#1}}}
\newcommand{\vj}[1]{{\bm{\mathsf{z}}_{\bm{j}#1}}}
\newcommand{\uj}[1]{{\bm{\mathsf{x}}_{\bm{j}#1}}}

\newcommand{\Qalt}{\bm{\mathtt{Q}}}
\newcommand{\M}{\mathcal{M}}

\newcommand{\MMHS}{{\mathrm{MMHS}}}

\DeclareMathOperator{\Hilb}{Hilb}

\DeclareMathOperator{\Spec}{Spec}
\DeclareMathOperator{\ch}{ch}

\DeclareMathOperator{\Ell}{Ell}
\DeclareMathOperator{\EllClass}{\mathcal{E}\!\mathit{ll}}

\DeclareMathOperator{\MMHM}{\mathrm{MMHM}}

\DeclarePairedDelimiter\abs{\lvert}{\rvert}

\begin{document}

% Make title, abstract etc.-----------------------------------------------------------------------------------------------------------------------------------------------------
\title[Instantons and multibananas]{Instantons and multibananas: Relating elliptic genus and cohomological Donaldson-Thomas theory}
\author{Oliver Leigh}
\email{oliverle@math.uio.no}
\address{Dept. of Mathematics, The University of Oslo}

%\date{Apr 2024}

\begin{abstract}
In this article the cohomological Donaldson-Thomas theory of local multibanana threefolds is computed using Descombes' hyperbolic localisation formula. 
The resulting expression is precisely given by the elliptic genus of the moduli space of framed instanton sheaves. 
As part of the proof, an infinite-wedge-space trace formula is computed for the elliptic genus of the moduli space of framed instanton sheaves.
\end{abstract}

\thanks{The author was supported by the Research Council of Norway grant number 302277 - \textit{``Orthogonal gauge duality and non-commutative geometry''.}}

\maketitle

% Introduction------------------------------------------------------------------------------------------------------------------------------------------------------------------

\section{Introduction}

\subsection{Cohomological Donaldson-Thomas invariants}

    ~~Numerical Donaldson-Thomas (DT) invariants are a mathematical count of BPS states; in particular supersymmetric bound states in type II string compactifications. 
    They were originally developed to study compact Calabi-Yau threefolds, for which they contain important deformation-invariant information, and have since been studied in much more generality \cite[\S3]{Thomas_Holomorphic}, \cite{MNOP_II}.

    Let $Y$ be a (not necessarily proper) Calabi-Yau 3-fold over $\CC$, and consider a curve class $\beta \in H_2(Y,\ZZ)$ and an integer $n\in\ZZ$. The \textit{numerical DT invariant} $\mathsf{DT}_{\beta,n}(Y)$ is computed by considering the Hilbert scheme
    \[
        \Hilb^{_{\beta,n}}(Y) 
        := 
        \{ Z\subset Y ~~|~\, [Z]=\beta \mbox{ and } \chi(\O_{Z}) = n \}. 
    \]
    The invariant $\mathsf{DT}_{\beta,n}(Y)$ is given by a weighted Euler characteristic of $\Hilb^{_{\beta,n}}(Y)$ \cite{Behrend_DonaldsonThomas} or, equivalently if $Y$ is proper, a virtual fundamental class of $\Hilb^{_{\beta,n}}(Y)$ \cite{Thomas_Holomorphic}. 

    In recent years, the research group of Joyce has developed a rigorous general theory of  \emph{cohomological DT invariants} in \cite{Joyce_Classical,BBJ_Darboux,BBDJS_Symmetries,BBBJ_Darboux}, based on ideas due to Kontsevich and Soibelman in \cite{KS_Stability,KS_Cohomological}, and Dimca and Szendrői in  \cite{DS_Milnor} (see also \cite{Davison_Critical,DM_Cohomological}, and for an excellent introduction to the field \cite{Szendroi_Cohomological}).
    Cohomological DT invariants provide more nuanced information, at the expense of deformation invariance (see \cite[\S8.1]{Szendroi_Cohomological} for a discussion), and they are valued in the abelian category of monodromic mixed Hodge structures, $\MMHS = \MMHM(\Spec\CC)$. 

    The key result needed to define cohomological DT invariants in the present article is:

    \begin{thm}
        [{\cite[Cor. 6.12]{BBDJS_Symmetries} and \cite{BD_Relative}}]
        \label{thm: MMHS existence}
        Let $\phi:\calE^{\bullet}\rightarrow L^{\bullet}_{\Hilb^{_{\beta,n}}(Y)}$ be the natural symmetric obstruction theory of Thomas \cite[Thm. 3.30]{Thomas_Holomorphic}. 
        For a given choice of square root $\det(\calE^\bullet)^{1/2}$ there is a monodromic mixed hodge module
        $\mathcal{DT}^{\bullet}_{Y,\beta,n} \in\MMHM(\Hilb^{_{\beta,n}}(Y))$, which is unique up to canonical isomorphism, satisfying:
        \begin{itemize}[leftmargin=2em]
            \item 
            If $\Hilb^{_{\beta,n}}(Y)$ is locally modelled by the critical locus of a regular function $f : U \rightarrow \CC$  on a smooth $\CC$-scheme $U$, then 
            $\mathcal{DT}^{\bullet}_{Y,\beta,n}$ is locally modelled by 
            $\mathcal{HV}^\bullet_{U,f}$, the monodromic mixed Hodge module lift of the perverse sheaf of vanishing cycles of $(U,f)$.
        \end{itemize}
    \end{thm}

    \begin{rem}
        The results of \cite{BBDJS_Symmetries} apply to more general spaces than those given in Theorem \ref{thm: MMHS existence}. 
        We have restricted to the case at hand in order to simplify the exposition. 
    \end{rem}

    \begin{rem}
        Nekrasov and Okounkov show in \cite[\S6]{NO_Membranes} (see also \cite[Prop. 2.1]{Arbesfeld_K-Theoretic_DT}) that a choice of square root $\det(\calE^\bullet)^{1/2}$ for Theorem \ref{thm: MMHS existence} always exists. 
    \end{rem}

    We call $\mathcal{DT}^{\bullet}_{Y,\beta,n}$ the \emph{DT mixed hodge module}.
    As discussed in \cite[Rem. 2.22]{BBDJS_Symmetries}, the hypercohomology of a monodromic mixed Hodge module also has the structure of a monodromic mixed Hodge structure.
    Hence, using the monodromic mixed Hodge module from Theorem \ref{thm: MMHS existence} one can define cohomological DT invariants as follows:

    \begin{defn}
        The \emph{cohomological Donaldson-Thomas invariants} are defined by 
        \[
            \mathbb{DT}_{\beta,n}(Y)
            :=
            \HH^\bullet\left(\mathcal{DT}^{\bullet}_{Y,\beta,n}\right)
            \in K(\MMHS)
        \]
        where $\HH^\bullet(-)$ is compactly supported hypercohomology. 
    \end{defn}

    \begin{rem}
        As discussed in \cite{Behrend_DonaldsonThomas} and \cite[Remark 6.14]{BBDJS_Symmetries}, the numerical DT invariant can be recovered from the cohomological DT invariant by 
        \[
            \mathsf{DT}_{\beta,n}(X)
            =
            \sum_{k\in\ZZ}
            (-1)^k
            \dim
            \HH^k\left(\mathcal{DT}^{\bullet}_{Y,\beta,n}\right).
        \]
    \end{rem}

    Similar to the case of numerical DT invariants, one can assemble the cohomological DT invariants into a partition function
    \[
        Z^{\mathbb{DT}}(Y)
        :=
        \sum_{\beta\in H_2(Y,\ZZ)}
        \sum_{n\in\ZZ}
        ~\mathbb{DT}_{\beta,n}(Y) \, Q^{\beta}\,q^{n},
    \]
    where $Z^{\mathbb{DT}}(Y)$ is an element of $K(\MMHS)[q^{-1}]\psl q, Q \psr$. 
    Here an effective basis $\beta_1,\ldots,\beta_n$ of $H_2(Y,\ZZ)$ has been chosen i.e. if $\beta$ is effective then $\beta=\sum d_i\beta_i$ for $d_i \geq0$, and the convention $Q^{\beta} = \prod Q_i^{d_i\beta_i}$ has been used. 

    The above partition function is very hard to compute. Indeed, even for the easier numerical DT version, there is no complete computation for any proper Calabi-Yau threefold with trivial first Betti number. The case is worse for cohomological DT theory, but the recent \emph{hyperbolic localisation formula} of \cite{Descombes_Hyperbolic} (discussed more in Section \ref{sec: DT computation hyperbolic}) has provided a powerful tool for calculations. 
    In the present article, we use this tool to compute the cohomological DT theory of local multibanana Calabi-Yau threefolds.

    Most previous calculations in this area have been from the view-point of \textit{motivic DT invariants}, introduced in \cite{KS_Stability}. These take values in a Grothendieck ring of varieties, but can be converted to cohomological DT theory as described in \cite[\S7.10]{KS_Cohomological} and \cite[App. A]{Davison_Refined}.
    From this point of view, the partition function $Z^{\mathbb{DT}}(Y)$ has been computed for local curves in \cite{MMNS_Motivic,MN_Motivic,DM_Motivic_-2,Descombes_Cohomological} (see \cite[\S5]{MP_Attractor} for a review), and the invariants for the one-loop
    quiver with potential were calculated in 
    \cite{DM_Motivic_Loop}. 
    The theory relating to some local surfaces has also been computed in \cite{BMP_VafaWitten,MP_Attractor,Descombes_Cohomological}.

\subsection{Local multibanana threefolds}
\label{sec: multibanana properties}

Local multibanana threefolds are examples of non-singular toric varieties without a finite type fan. Their name is derived from the shape of a curve configuration which is discussed below and depicted in Figure \ref{fig: banana diagrams}.

Following Kanazawa and Lau \cite[\S5]{KL_Local}, for $i,j\in\NN$, one can construct the local multibanana threefold of type $(i,j)$  as a quotient of a toric threefold with a non-finite-type fan. 
We can construct this fan by first considering the lines $\mathsf{L}_{f,m}:=\big\{ p\in\CC^3 \,|\, f(p)=m,z=1\big\}$ and the tiling of $\{ (x,y,1)  \,|\, x,y\in\CC\}$ given by the collection:
\[
    \big\{  \mathsf{L}_{x,l} \big\}_{l\in \ZZ}
    \cup
    \big\{  \mathsf{L}_{y,m} \big\}_{m\in \ZZ}
    \cup
    \big\{  \mathsf{L}_{x-y,n} \big\}_{n\in \ZZ}.
\]
This tiling also defines a non-finite-type fan $F$ by taking the all cones over the proper faces of the tiling. This is depicted in Figure \ref{fig: fan and vertices} below. 

\begin{figure}[t]
    \hfill
    \raisebox{-0.5\height}{
        \begin{tikzpicture}[
            scale=.8,
            every node/.style={minimum size=1cm},
            on grid
            ] 
            \draw[black!20] (-0.2*8.66, -0.075*5) -- (-0.0*8.66, 0.0*5);
            \draw[black!20] (-0.2*8.66, -0.075*5) -- (-0.2*8.66, 0.2*5);
            \draw[black!20] (-0.2*8.66, -0.075*5) -- (-0.4*8.66, 0.4*5);
            \draw[black!20] (-0.2*8.66, -0.075*5) -- ( 0.2*8.66, 0.2*5);
            \draw[black!20] (-0.2*8.66, -0.075*5) -- ( 0.0*8.66, 0.4*5);
            \draw[black!20] (-0.2*8.66, -0.075*5) -- (-0.2*8.66, 0.6*5);
            \draw[black!20] (-0.2*8.66, -0.075*5) -- ( 0.4*8.66, 0.4*5);
            \draw[black!20] (-0.2*8.66, -0.075*5) -- ( 0.2*8.66, 0.6*5);
            \draw[black!20] (-0.2*8.66, -0.075*5) -- ( 0.0*8.66, 0.8*5);
            \draw[white, very thick] (-0.1*8.66,-0.1*5) -- ( 0.5*8.66, 0.5*5);
            \draw[white, very thick] (-0.3*8.66, 0.1*5) -- ( 0.3*8.66, 0.7*5);
            \draw[white, very thick] (-0.5*8.66, 0.3*5) -- ( 0.1*8.66, 0.9*5);
            \draw[white, very thick] ( 0.1*8.66,-0.1*5) -- (-0.5*8.66, 0.5*5);
            \draw[white, very thick] ( 0.3*8.66, 0.1*5) -- (-0.3*8.66, 0.7*5);
            \draw[white, very thick] ( 0.5*8.66, 0.3*5) -- (-0.1*8.66, 0.9*5);
            \draw[white, very thick] (-0.1*8.66, 0.0*5) -- ( 0.1*8.66, 0.0*5);
            \draw[white, very thick] (-0.3*8.66, 0.2*5) -- ( 0.3*8.66, 0.2*5);
            \draw[white, very thick] (-0.5*8.66, 0.4*5) -- ( 0.5*8.66, 0.4*5);
            \draw[white, very thick] (-0.3*8.66, 0.6*5) -- ( 0.3*8.66, 0.6*5);
            \draw[white, very thick] (-0.1*8.66, 0.8*5) -- ( 0.1*8.66, 0.8*5);
            \draw[black, thick] (-0.1*8.66,-0.1*5) -- ( 0.5*8.66, 0.5*5);
            \draw[black, thick] (-0.3*8.66, 0.1*5) -- ( 0.3*8.66, 0.7*5);
            \draw[black, thick] (-0.5*8.66, 0.3*5) -- ( 0.1*8.66, 0.9*5);
            \draw[black, thick] ( 0.1*8.66,-0.1*5) -- (-0.5*8.66, 0.5*5);
            \draw[black, thick] ( 0.3*8.66, 0.1*5) -- (-0.3*8.66, 0.7*5);
            \draw[black, thick] ( 0.5*8.66, 0.3*5) -- (-0.1*8.66, 0.9*5);
            \draw[black, thick] (-0.1*8.66, 0.0*5) -- ( 0.1*8.66, 0.0*5);
            \draw[black, thick] (-0.3*8.66, 0.2*5) -- ( 0.3*8.66, 0.2*5);
            \draw[black, thick] (-0.5*8.66, 0.4*5) -- ( 0.5*8.66, 0.4*5);
            \draw[black, thick] (-0.3*8.66, 0.6*5) -- ( 0.3*8.66, 0.6*5);
            \draw[black, thick] (-0.1*8.66, 0.8*5) -- ( 0.1*8.66, 0.8*5);
\end{tikzpicture}
    }
    \hfill
    \raisebox{-0.5\height}{
        \begin{tikzpicture}[
            scale=0.8,
            every node/.style={minimum size=1cm},
            on grid
            ] 
            % 1st hexagon
            \draw[A, ultra thick]     (0*1 + 0*0.707, 0*1 + 0*0.707) -- node[left,xshift=0.25cm] {\small$A_0$} ( 0*1 + 0*0.707, 1*1 + 0*0.707);
            \draw[C, ultra thick]   ( 0*1 + 0*0.707, 0*1 + 0*0.707) -- node[below,yshift=0.25cm] {\small$C_0$} ( 1*1 + 0*0.707, 0*1 + 0*0.707);
            \draw[B, ultra thick]    ( 0*1 + 0*0.707, 1*1 + 0*0.707) -- node[above,yshift=-0.3cm,xshift=-0.15cm] {\small$B_1$} ( 0*1 + 1*0.707, 1*1 + 1*0.707);
            \draw[B, ultra thick]    ( 1*1 + 0*0.707, 0*1 + 0*0.707) -- node[below,yshift=0.3cm,xshift=0.15cm] {\small$B_0$} ( 1*1 + 1*0.707, 0*1 + 1*0.707);
            \draw[A, ultra thick]     ( 1*1 + 1*0.707, 0*1 + 1*0.707) -- node[right,xshift=-0.25cm] {\small$A_0$} ( 1*1 + 1*0.707, 1*1 + 1*0.707);
            \draw[C, ultra thick]   ( 0*1 + 1*0.707, 1*1 + 1*0.707) -- node[below,yshift=0.2cm] {\small$C_1$} ( 1*1 + 1*0.707, 1*1 + 1*0.707);
            %
            % 2nd hexagon
            \draw[A, ultra thick]     ( 0*1 + 1*0.707, 1*1 + 1*0.707) -- node[left,xshift=0.25cm] {\small$A_1$} ( 0*1 + 1*0.707,  2*1 + 1*0.707);
            \draw[B, ultra thick]    (  0*1 + 1*0.707, 2*1 + 1*0.707) -- node[above,yshift=-0.3cm,xshift=-0.15cm] {\small$B_2$} ( 0*1 + 2*0.707, 2*1 + 2*0.707);
            \draw[B, ultra thick]    ( 1*1 + 1*0.707, 1*1 + 1*0.707) -- node[below,yshift=0.3cm,xshift=0.15cm] {\small$B_1$} ( 1*1 + 2*0.707, 1*1 + 2*0.707);
            \draw[A, ultra thick]     ( 1*1 + 2*0.707, 1*1 + 2*0.707) -- node[right,xshift=-0.25cm] {\small$A_1$} ( 1*1 + 2*0.707, 2*1 + 2*0.707);
            \draw[C, ultra thick]   ( 0*1 + 2*0.707, 2*1 + 2*0.707) -- node[below,yshift=0.2cm] {\small$C_2$} ( 1*1 + 2*0.707,  2*1 + 2*0.707);
            %
            % 3rd hexagon
            \draw[A, ultra thick]     ( 0*1 + 2*0.707, 2*1 + 2*0.707) --  node[left,xshift=0.25cm] {\small$A_2$} ( 0*1 + 2*0.707, 3*1 + 2*0.707);
            \draw[B, ultra thick]    ( 0*1 + 2*0.707, 3*1 + 2*0.707) -- node[above,yshift=-0.3cm,xshift=-0.15cm] {\small$B_0$} ( 0*1 + 3*0.707, 3*1 + 3*0.707);
            \draw[B, ultra thick]    ( 1*1 + 2*0.707, 2*1 + 2*0.707) -- node[below,yshift=0.3cm,xshift=0.15cm] {\small$B_2$} ( 1*1 + 3*0.707, 2*1 + 3*0.707);
            \draw[A, ultra thick]     ( 1*1 + 3*0.707, 2*1 + 3*0.707) -- node[right,xshift=-0.25cm] {\small$A_2$} ( 1*1 + 3*0.707, 3*1 + 3*0.707);
            \draw[C, ultra thick]   ( 0*1 + 3*0.707, 3*1 + 3*0.707) -- node[below,yshift=0.2cm] {\small$C_0$} ( 1*1 + 3*0.707, 3*1 + 3*0.707);
            %
            % 1st Half-legs
            \draw[B, ultra thick]     ( 0*1 + 0*0.707, 0*1 + 0*0.707) -- node[left,yshift=-0.15cm,xshift=0.1cm] {\small$B_0$} ( 0*1 - 0.5*0.707, 0*1 - 0.5*0.707);
            \draw[C, ultra thick]   ( 0*1 + 0*0.707, 1*1 + 0*0.707) -- node[above,yshift=-0.2cm,xshift=-0.2cm] {\small$C_1$} ( -0.5*1 + 0*0.707, 1*1 + 0*0.707);
            \draw[C, ultra thick]   ( 1*1 + 1*0.707, 0*1 + 1*0.707) -- node[below,yshift=0.2cm,xshift=0.2cm] {\small$C_0$} ( 1.5*1 + 1*0.707, 0*1 + 1*0.707);
            \draw[A, ultra thick]     ( 1*1 + 0*0.707, 0*1 + 0*0.707) -- node[right,yshift=-0.07cm,xshift=-0.25cm] {\small$A_2$} ( 1*1 + 0*0.707, -0.5*1 + 0*0.707);
            %
            % 2nd Half-legs
            \draw[C, ultra thick]   ( 0*1 + 1*0.707, 2*1 + 1*0.707) -- node[above,yshift=-0.2cm,xshift=-0.2cm] {\small$C_2$} ( -0.5*1 + 1*0.707, 2*1 + 1*0.707);
            \draw[C, ultra thick]   ( 1*1 + 2*0.707, 1*1 + 2*0.707) -- node[below,yshift=0.2cm,xshift=0.2cm] {\small$C_1$} ( 1.5*1 + 2*0.707, 1*1 + 2*0.707);
            %
            % 3rd Half-legs
            \draw[C, ultra thick]   ( 0*1 + 2*0.707, 3*1 + 2*0.707) -- node[above,yshift=-0.2cm,xshift=-0.2cm] {\small$C_0$} ( -0.5*1 + 2*0.707, 3*1 + 2*0.707);
            \draw[C, ultra thick]   ( 1*1 + 3*0.707, 2*1 + 3*0.707) -- node[below,yshift=0.2cm,xshift=0.2cm] {\small$C_2$} ( 1.5*1 + 3*0.707, 2*1 + 3*0.707);
            \draw[A, ultra thick]     ( 0*1 + 3*0.707, 3*1 + 3*0.707) -- node[right,yshift=0.1cm,xshift=-0.25cm] {\small$A_1$} ( 0*1 + 3*0.707, 3.5*1 + 3*0.707);
            \draw[B, ultra thick]     ( 1*1 + 3*0.707, 3*1 + 3*0.707) -- node[right,yshift=0.17cm,xshift=-0.1cm] {\small$B_0$} ( 1*1 + 3.5*0.707, 3*1 + 3.5*0.707);
            %
            % Vertices
            \node[fill,circle,minimum size=3.2pt,inner sep=0,] at ( 0*1 + 0*0.707, 0*1 + 0*0.707) {};
            \node[fill,circle,minimum size=3.2pt,inner sep=0,] at ( 0*1 + 0*0.707, 1*1 + 0*0.707) {};
            \node[fill,circle,minimum size=3.2pt,inner sep=0,] at ( 1*1 + 0*0.707, 0*1 + 0*0.707) {};
            \node[fill,circle,minimum size=3.2pt,inner sep=0,] at ( 0*1 + 1*0.707, 1*1 + 1*0.707) {};
            \node[fill,circle,minimum size=3.2pt,inner sep=0,] at ( 1*1 + 1*0.707, 0*1 + 1*0.707) {};
            \node[fill,circle,minimum size=3.2pt,inner sep=0,] at ( 1*1 + 1*0.707, 1*1 + 1*0.707) {};
            \node[fill,circle,minimum size=3.2pt,inner sep=0,] at ( 0*1 + 1*0.707, 2*1 + 1*0.707) {};
            \node[fill,circle,minimum size=3.2pt,inner sep=0,] at ( 1*1 + 2*0.707, 1*1 + 2*0.707) {};
            \node[fill,circle,minimum size=3.2pt,inner sep=0,] at ( 0*1 + 2*0.707, 2*1 + 2*0.707) {};
            \node[fill,circle,minimum size=3.2pt,inner sep=0,] at ( 1*1 + 2*0.707, 2*1 + 2*0.707) {};
            \node[fill,circle,minimum size=3.2pt,inner sep=0,] at ( 0*1 + 2*0.707, 3*1 + 2*0.707) {};
            \node[fill,circle,minimum size=3.2pt,inner sep=0,] at ( 1*1 + 3*0.707, 2*1 + 3*0.707) {};
            \node[fill,circle,minimum size=3.2pt,inner sep=0,] at ( 0*1 + 3*0.707, 3*1 + 3*0.707) {};
            \node[fill,circle,minimum size=3.2pt,inner sep=0,] at ( 1*1 + 3*0.707, 3*1 + 3*0.707) {};
\end{tikzpicture}
    }
    \hfill
    \hfill
    \caption{On the left is part of the non-finite type toric fan of $\calX$, and on the right is the web diagram of the local multibanana threefold $X_{(3,1)}$.}
    \label{fig: fan and vertices}
\end{figure}
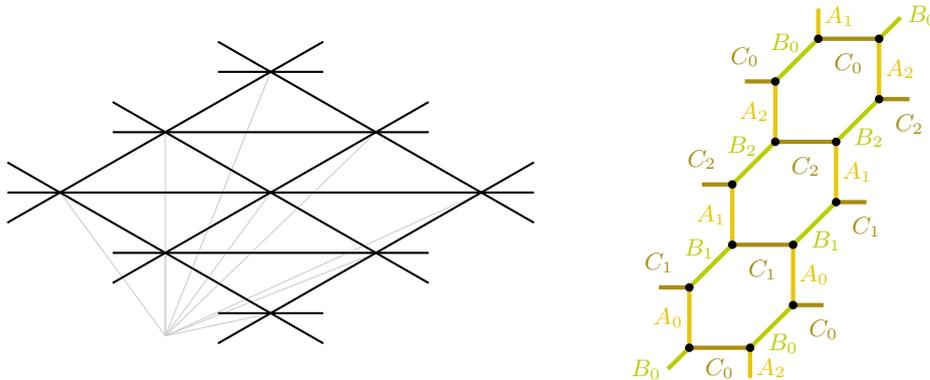

Now, take the threefold $\calX$ defined by $F$ and observe that $\calX$ has a free $i\ZZ\oplus j\ZZ$-action coming from translations of $F$. 
We define \emph{the local multibanana threefold of type $(i,j)$} by
\[
    X_{(i,j)} 
    :=
    \calX/i\ZZ\oplus j\ZZ.
\]
The focus of this article is $X_{(r,1)}$, and for convenience, we restrict to this space from now on. $X_{(r,1)}$ has a natural $\TT^2$ action and, as explained in \cite[\S5.1]{KL_Local}, the $\TT^2$-fixed curves are unions of $(-1,-1)$-rational curves corresponding to 2-dimensional sub-cones. Following Bryan \cite{Bryan_DonaldsonThomas}, we call these the \emph{banana curves} and label them by their sub-cones
\begin{align*}
    A_k 
    &\sim
    \mathbb{R}_{\geq0} \mathrm{Conv}\big(\{(0,k,1),(0,k + 1,1)\}\big),
    \\
    B_k
    &\sim
    \mathbb{R}_{\geq0} \mathrm{Conv}\big(\{(k+1,k,1),(k,k + 1,1)\}\big),
    \\
    C_k
    &\sim
    \mathbb{R}_{\geq0} \mathrm{Conv}\big(\{(k,0,1),(k+1,0,1)\}\big).
\end{align*}
%These labels are depicted in the web diagram of Figure \ref{fig: fan and vertices} above.
As pointed out in \cite[\S5.1]{KL_Local}, the banana curves have the relations 
\[
    A_k + B_k = A_k + B_{k+1}
    \hspace{1cm}
    \mbox{and}
    \hspace{1cm}
    B_k + C_K = B_{k+1} + C_{K+1}
\]
in $H_2\big(X_{(r,1)},\ZZ\big)$, and hence the banana curves form the  r+2 classes
\[
    A_0,\ldots,A_{r-1},
    \hspace{1cm}
    B = B_0,\ldots,B_{r-1},
    \hspace{1cm}
    \mbox{and}
    \hspace{1cm}
    C = C_0,\ldots,C_{r-1}.
\]
Moreover, by \cite[Lem. 5.3]{KL_Local} these classes give an effective basis for $H_2\big(X_{(r,1)},\ZZ\big)$.

The banana curves are either disjoint or meet as the coordinate axis in a copy of $\CC^3$. This is depicted in the web diagram of Figure \ref{fig: fan and vertices} and the banana diagrams in Figure \ref{fig: banana diagrams}. We call the union of the banana curves a \emph{banana configuration}.
Interestingly, the mirror symmetry of Banana configurations (called perverse curves) was previously studied by Ruddat in \cite{Ruddat_Perverse}.

\begin{figure}
    \raisebox{-0.5\height}{
        \begin{tikzpicture}[
    scale=0.85,
    every node/.style={minimum size=1cm},
    on grid
    ]
    \clip (-2.5,-0.6) rectangle (2.5, 4.13);
    \draw[A, ultra thick]  (-1, 0) -- node[below,yshift=0.26cm] {\small$A_0$} ( 1, 0);
    \draw[A, ultra thick]  ( 1 + 1.8*0.5, 0 + 1.8*0.866) -- node[right,xshift=-0.2cm] {\small$A_1$} ( 1 - 0.2*0.5, 0 + 3.8*0.866);
    \draw[A, ultra thick]  (-1 - 1.8*0.5, 0 + 1.8*0.866)  -- node[left,xshift=0.2cm] {\small$A_2$} (-1 +
    0.2*0.5, 0 + 3.8*0.866);
    \draw[B, ultra thick] (1 , 0)  arc[
                        start angle=330-45.0631,
                        end angle=330+45.0631,
                        radius=1.271,
                    ] node[right,yshift=-1.0cm,xshift=-0.3cm] {\small$B_0$} ;
    \draw[C, ultra thick] ( 1 + 1.8*0.5, 0 + 1.8*0.866)  arc[
                        start angle=150-45.0631,
                        end angle=150+45.0631,
                        radius=1.271,
                    ] node[left,yshift=1.0cm,xshift=0.3cm] {\small$C_0$};
    \draw[B, ultra thick] ( 1 - 0.2*0.5, 0 + 3.8*0.866)  arc[
                    start angle=90-45.0631,
                    end angle=90+45.0631,
                    radius=1.271,
                ] node[above,yshift=0.0cm,xshift=0.8cm] {\small$B_1$};
    \draw[C, ultra thick] (-1 + 0.2*0.5, 0 + 3.8*0.866)  arc[
                    start angle=270-45.0631,
                    end angle=270+45.0631,
                    radius=1.271,
                ] node[below,yshift=-0.2cm,xshift=-0.9cm] {\small$C_1$};
    \draw[B, ultra thick] (-1 - 1.8*0.5, 0 + 1.8*0.866)  arc[
                        start angle=210-45.0631,
                        end angle=210+45.0631,
                        radius=1.271,
                    ] node[left,yshift=0.56cm,xshift=-0.5cm] {\small$B_2$};
    \draw[C, ultra thick] (-1 , 0)  arc[
                        start angle=30-45.0631,
                        end angle=30+45.0631,
                        radius=1.271,
                    ] node[right,yshift=-0.55cm,xshift=0.6cm] {\small$C_2$};
    \node[fill,circle,minimum size=3.2pt,inner sep=0,] at (-1, 0) {};
    \node[fill,circle,minimum size=3.2pt,inner sep=0,] at ( 1, 0) {};
    \node[fill,circle,minimum size=3.2pt,inner sep=0,] at ( 1 + 1.8*0.5, 0 + 1.8*0.866) {};
    \node[fill,circle,minimum size=3.2pt,inner sep=0,] at ( 1 - 0.2*0.5, 0 + 3.8*0.866) {};
    \node[fill,circle,minimum size=3.2pt,inner sep=0,] at (-1 - 1.8*0.5, 0 + 1.8*0.866) {};
    \node[fill,circle,minimum size=3.2pt,inner sep=0,] at (-1 +
     0.2*0.5, 0 + 3.8*0.866) {};
    %\draw[gray,step=0.25] (-2.5,-0.5) grid (2.5,4.5);
\end{tikzpicture}
    }
    \hspace*{1cm}
    \raisebox{-0.5\height}{
        \includegraphics[width=3.5cm]{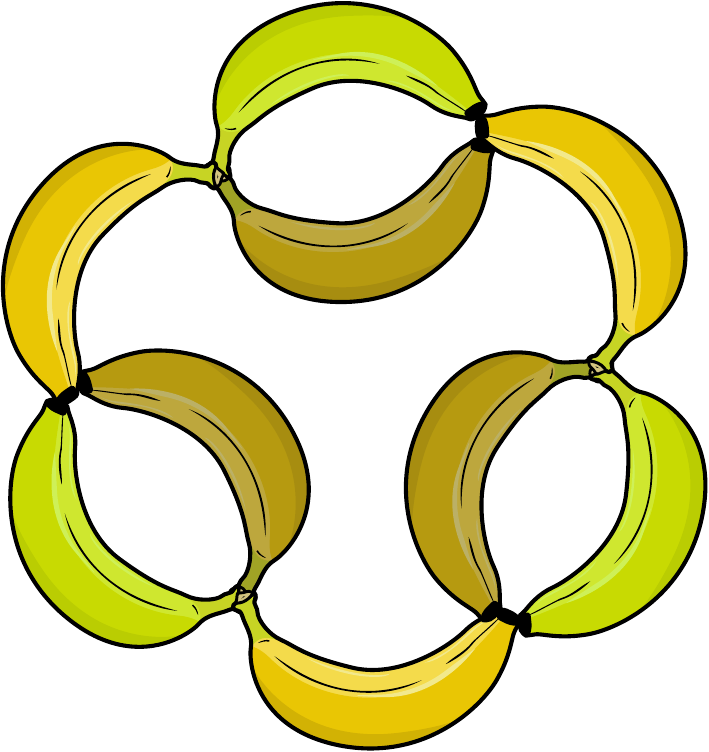}
    }
    \caption{Depictions of the \emph{banana configurations}. \small(Banana image from: Vecteezy.com)}
    \label{fig: banana diagrams}
\end{figure}
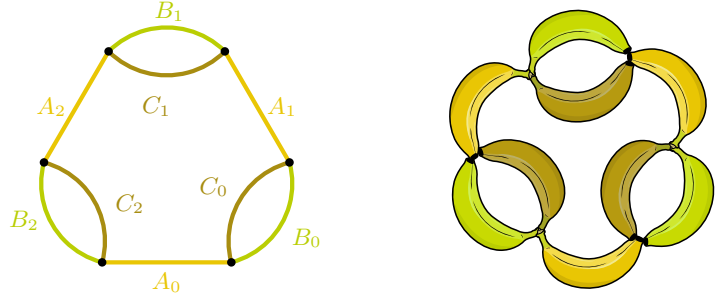

\subsection{Enumerative geometry of local multibanana threefolds}
\label{sec: enumerative geometry of multibananas}

The enumerative geometry of banana threefolds was first studied by Bryan in \cite{Bryan_DonaldsonThomas}. He considered an elliptically-fibred proper Calabi-Yau threefold $X_{\mathsf{ban}}$ containing subschemes locally isomorphic to $X_{(1,1)}$, then computes the numerical DT partition function for the banana-curve sublattice. 
We give a cohomological generalisation of the local version of Bryan's formula in Corollary \ref{cor: r=1 product}. 

Bryan's results were extended by the author in \cite{Leigh_Unweighted}, by also considering the section of the elliptic fibration. 
In forthcoming work, Bryan and Pietromonaco consider related quotient spaces called banana nano-manifolds in \cite{BP_Geometry}.
In another direction, Katz's genus 0 Gopakumar-Vafa partition function of $X_{\mathsf{ban}}$ for the banana-curve sublattice was partially computed in \cite{Morishige_Genus}. This calculation was extended to $X_{(r,1)}$ and $X_{(2,2)}$ in \cite{Morishige_Genus_Multi}. 

From the discussion in Section \ref{sec: multibanana properties} we know $A_0,\ldots,A_{r-1},B,C$ is an effective basis for $H_2(X_{(r,1)},\ZZ)$, 
and denote the corresponding formal variables by $Q_{A_0},\ldots,Q_{A_{r-1}},Q_B,Q_C$. 
We can also define a restricted partition function $Z^{\mathbb{DT}}_{c}(X_{(r,1)})$ by 
\[
    Z^{\mathbb{DT}}\big(X_{(r,1)}\big)
    =
    \sum_{c \geq 0}
    \,
    Z^{\mathbb{DT}}_{c}\big(X_{(r,1)}\big)
    \,
    Q_C^c.
\]
The main result of this article is the following explicit formula for this partition function. 

%\vspace{0.5em}

\begin{thm}
    [\textbf{Main Theorem}]
    \label{thm: main}
    Let $\LL^{1/2}$ be the monodromic mixed hodge module of $\CC \rightarrow \CC; z \mapsto z^2 $ as detailed in \cite[p. 803]{DM_Cohomological}, and let $\LL^{-1/2}$ be its formal inverse.  
    Define
    \[
        \Theta_{p,y}(t) :=
        \prod_{n>0} 
            \dfrac{(1 - y\,p^{n-1} t)}{(1 - p^{n-1} t)}
            \dfrac{(1 - y^{-1} p^{n} t^{-1})}{(1 - p^{n} t^{-1} )},
    \]
    the generalised MacMahon function
    $
        M(u; v, w) 
        :=
        \prod_{k,l > 0} ( 1 - u\, v^{k}\, w^{l})^{-1},
    $
    and for a 2D partition Y with $(i,j)\in Y$ the arm and leg lengths $a_Y(i,j) := Y_i - j$ and $l_Y(i,j) := (Y^T)_j - i$.

    Then, for any choice of orientation, we have that $Z^{\mathbb{DT}}_{c}\big(X_{(r,1)}\big)$ is equal to the expression
    \begin{align*}
        &
        \calZ_0^{\mathbb{DT}}
        \cdot
        Q_B^{(1-r)c} 
        \!\!\!
        \multisum{\gamma_0,\ldots,\gamma_{r-1} \\ s.t. \sum \gamma_i = c}
        \!\!\!\!\!
        \big(-u^{-\frac{1}{2}} {v}^{\frac{1}{2}}\big)^{rc}
        \prod_{i,j=0}^{r-1}
        \,
        \prod_{x \in \gamma_i}
        \,
        \Theta_{\mathtt{e}_{r-1},-u^{1/2} v^{-1/2} Q_B}
        \!
        \left( \dfrac{\mathtt{e}_i}{\mathtt{e}_j} u^{-a_{\gamma_i}(x)-1} v^{-l_{\gamma_j}(x)} \right)
        \\
        &
        \hspace{18.09em}
        \prod_{z \in \gamma_j}  
        \,
        \Theta_{\mathtt{e}_{r-1},-u^{1/2} v^{-1/2} Q_B} 
        \!
        \left( \dfrac{\mathtt{e}_i}{\mathtt{e}_j} u^{a_{\gamma_j}(z)} v^{l_{\gamma_i}(z)+1}  \right)
    \end{align*}
    where 
    $u:= - q \,\LL^{1/2}$, $v:= - q \,\LL^{-1/2}$ and
    $\mathtt{e}_i : = \prod_{j=0}^i Q_{A_i} Q_B$, 
    and $\calZ_{0}^{\mathbb{DT}}$ is given by the expression
    \begin{align*}
        %\calZ_0^{\mathbb{DT}}
        %=
        &
        ~
        M(v^{-1};u,v)^{r}
        \cdot
        \prod_{0\leq i<j\leq r-1}
        \frac{
            M\big(\mathtt{e}_i\,\mathtt{e}_j^{-1}v^{-1}; u,v\big)
        }{
            M\big(\mathtt{e}_j\,\mathtt{e}_i^{-1}u^{-1}; u,v\big)
        }
        \frac{M\big( -Q_B\, \mathtt{e}_j\,\mathtt{e}_i^{-1}(uv)^{-\frac{1}{2}}; u,v\big)}{M\big( -Q_B^{-1}  \mathtt{e}_i\,\mathtt{e}_j^{-1} (uv)^{-\frac{1}{2}}; u,v\big)}
        \\
        &
        ~
        \cdot
        \prod_{N>0}
        \frac{1}{1-e_{r-1}^N}
        \prod_{i,j=0}^{r-1} 
        \frac{M\big(\mathtt{e}_{r-1}^{N-1}\mathtt{e}_i\,\mathtt{e}_j^{-1}u^{-1}; u,v\big)}{M\big(-Q_B\,\mathtt{e}_{r-1}^{N-1} \mathtt{e}_i\,\mathtt{e}_j^{-1} (uv)^{-\frac{1}{2}}; u,v\big)}
        \frac{M\big(\mathtt{e}_{r-1}^{N}\mathtt{e}_i\,\mathtt{e}_j^{-1}v^{-1}; u,v\big)}{M\big(-Q_B^{-1}\mathtt{e}_{r-1}^{N} \mathtt{e}_i\,\mathtt{e}_j^{-1} (uv)^{-\frac{1}{2}}; u,v\big)}.
    \end{align*}
\end{thm}

\begin{rem}
    Another interesting refinement of DT invariants is given by the K-theoretic DT invariants of \cite{NO_Membranes} (see also \cite{Arbesfeld_K-Theoretic_DT,Arbesfeld_K-Theoretic_Descendent}).
    In the case when $Y$ is proper, these are predicted to be the $\chi_y$-genus of the cohomological DT invariants. 
    However, in the case where $Y$ is non-proper, there is a discrepancy. 
    An example is when $Y=\CC^3$ and $\beta=0$ as computed motivically by \cite{BBS_Motivic} and K-theoretically by \cite[\S8.3]{NO_Membranes}. This relationship is discussed in \cite[\S4.4 and \S5.1]{Descombes_Cohomological}. 

    Another discrepancy is for $X_{(r,1)}$. K-theoretic DT invariants use a refinement parameter arising from the equivariant weight of the canonical bundle. However, since $\TT^2$ is a Calabi-Yau torus on $X_{(r,1)}$, this refinement parameter is $1$, and no extra information is obtained in this case. 
\end{rem}

\subsection{Elliptic genus of the moduli space of framed instanton sheaves}

Theorem \ref{thm: main} has a striking relationship with the elliptic genus of framed instanton sheaves, which we will discuss in this section. 
We first recall the definition and some key properties of the moduli space of framed instanton sheaves on $\PP^2$. 

\begin{defnthm}[\cite{NY_Instanton_I}]
    \label{defthm: M(r,N)}
    The \emph{moduli space of frame instanton sheaves on $\PP^2$} $\M(r,N)$ is the nonsingular scheme of dimension $2rN$ parameterising pairs
    \[
        \Big(\, E,\, \Phi: E|_{\ell_\infty} \overset{\sim}{\longrightarrow} \O_{\ell_\infty}^{\oplus r}  \,\Big)
    \]
    where $E$ is a rank $r$ torsion free sheaf on $\PP^2$ with $\left<c_2(E),[\PP^2]\right> = n$ and that is locally free in a neighborhood of $\ell_\infty:=\{[0:z_1:z_2]\}$,
    and $\Phi$ is an isomorphism. It has the following properties:
    \begin{enumerate}
        \item 
        $\M(r,N)$ is acted upon by a $2+r$-dimensional torus, $\EE^{2+r}$, with associated characters $t_1,t_2,e_0,\ldots, e_{r-1}$. 
        \item 
        The fixed points $\M(r,N)^{\EE^{2+r}}$ are isolated and are in bijection to the set of $r$-tuples of Young diagrams $\bm{Y}=(Y_1,\ldots,Y_r)$ satisfying $\sum_{i=1}^r\abs{Y_i} = N$.
        \item 
        For a fixed point $\bm{Y}$, the equivariant $K$-theory class of the restricted tangent bundle $T_{\M(r,N)}|_{\bm{Y}}$ is given by 
        \[
            \sum_{a,b = 0}^{r-1} e_{b}\,e_{a}^{-1}\cdot \left(\sum_{s\in Y_{a}}\left(t_1^{-l_{Y_{b}}(s)}t_2^{a_{Y_{a}}(s)+1}\right) + \sum_{t\in Y_{b}} \left(t_1^{l_{Y_{a}}(t)+1} t_2^{-a_{Y_{b}}(t) }\right) \right).
        \]
    \end{enumerate}
\end{defnthm}

Recall that elliptic genus can be defined by taking the Chern roots $L_i$ of a tangent bundle and considering the multiplicative class
\begin{align*}
    \EllClass_{u,y}(L) 
    ~:=~
    c_1(L) \frac{\theta_1(u, y \ch(-L))}{\theta_1(u, \ch(-L))}
    ~=~
    \frac{c_1(L)}{y^{\frac{1}{2}}} \cdot \Theta_{u,y} \big(c_1(-L)\big).
\end{align*}
Now, combining Definition-Theorem \ref{defthm: M(r,N)} with Theorem \ref{thm: main} we have the following corollary.

\begin{cor}
    [\textbf{Partition function as elliptic genus}]
    \label{cor: Partition function as elliptic genus}
    Considering $\LL^{1/2}$ as a formal variable, we have the relationship
    \[
        \frac{Z^{\mathbb{DT}}_{c}\big(X_{(r,1)}\big)}{Q_B^c \cdot\calZ_{0}^{\mathbb{DT}}}
        =
        \mathrm{Ell}_{\mathtt{e}_{r-1},-\LL^{\frac{1}{2}}Q_B } 
        \left(\M(r,c); \mathtt{e}_0, \ldots, \mathtt{e}_{r-1}, \LL^{\frac{1}{2}}q^{-1}, \LL^{\frac{1}{2}} q  \right)
    \]
    where 
    $\mathtt{e}_i : = \prod_{j=0}^i Q_{A_i} Q_B$
    and $\calZ_{0}^{\mathbb{DT}}$ is given in Theorem \ref{thm: main}.
\end{cor}

\begin{rem}
    As part of the proof of Theorem \ref{thm: main} and Corollary \ref{cor: Partition function as elliptic genus} we obtain a natural expression for
    \[
        \mathrm{Ell}_{u,y} 
            \left(\M(r,N); e_0, \ldots, e_{r-1}, t_1, t_2  \right)
    \]
    as a trace of vertex operators of Okounkov \cite{Okounkov_Infinite} (see also \cite{OR_Random,Young_Generating,BKY_Trace}). 
    The expression is given in Corollary \ref{cor: elliptic genus as trace final}. 
    It is hoped that this will be useful in extending the $\chi_y$ Vafa-Witten Blowup formula of \cite{KLT_blowup} to elliptic genus. 
\end{rem}

The relationship between DT theory of banana threefolds and elliptic genus was first observed by Bryan in \cite{Bryan_DonaldsonThomas}, 
where it was observed that the numerical Donaldson-Thomas partition function for the local banana threefold $(r=1)$ was determined by the equivariant elliptic genus of the Hilbert scheme of points on $\CC^2$, where $\CC^*$ acts diagonally. 

Theorem \ref{thm: main} extends this result in two ways:
\begin{enumerate}
    \item For $r=1$ the cohomological DT partition function is determined by the equivariant elliptic genus of the Hilbert scheme of points on $\CC^2$ for the larger $(\CC^*)^2$ torus. 
    \item For $r>1$ the cohomological DT partition function is determined by the equivariant elliptic genus of the moduli space of framed instanton sheaves on $\PP^2$. 
\end{enumerate}

In the $r=1$ case we can follow Bryan \cite{Bryan_DonaldsonThomas} and use Waelder's analogue of the DMVV formula \cite[Thm. 12]{Waelder_Equivariant} to obtain an infinite product formula. 

\begin{cor}[\textbf{Product formula for the banana case, $r=1$}]
    \label{cor: r=1 product}
    For $r=1$ the full partition function is
    \begin{align*}
        Z^{\mathbb{DT}}
        =
        ~
        &
        \prod_{ d_A, d_B d_C \geq 0}
        \!
        \multiprod{\bm{k}=(k_1,k_2)\in (\frac{1}{2}\ZZ)^2 \\ \mbox{s.t. $\bm{k} \neq \bm{0}$ if $\bm{d} = \bm{0}$ }}
        \!\!\!\!
        \Big(1 - (-Q_{A_0})^{d_A} (-Q_B)^{d_B} (-Q_{C})^{d_C} v^{-k_1} u^{k_2}\Big)^{-b(\|\bm{d}\|,k_1,k_2)}.
    \end{align*}
    where 
    $\bm{d}: = (d_A,d_B,d_C)$, 
    $\|\bm{d}\| := 2d_A d_B + 2d_A d_C + 2d_B d_C - d_A^2 - d_B^2 - d_C^2$
    and $b(a,k_1,k_2)$ is defined by the expansion as a Laurent series in $p,y,t_1^{_{-1/2}},t_2^{_{1/2}}$ of
    \[
        \sum_{a\geq -1} \sum_{k_1,k_2\in \frac{1}{2}\ZZ}
        b(a,k_1,k_2) 
        Q^a t_1^{-k_1} t_2^{k_2}
        =
        \frac{
            \sum_{k\in \ZZ} Q^{k^2} \big(-(t_1t_2)^{\frac{1}{2}}\big)^k
        }{
            \left(
                \sum_{k\in \ZZ+\frac{1}{2}} Q^{2k^2} (-t_1)^{-k}
            \right)
            \left(
                \sum_{k\in \ZZ+\frac{1}{2}} Q^{2k^2} (-t_2)^k
            \right)
        }
    \]
\end{cor}

\begin{rem}
    It is unclear whether such a product formula exists in the general case.
    On the positive side: The existence of a product formula is equivalent to the existence of the associated Gopakumar-Vafa-invariant formula (see \cite[Def.-Thm. A.13]{Bryan_DonaldsonThomas} for more details). 

    Given the link with elliptic genus of $\M(r,n)$, one would hope that such a product formula extends to the elliptic genus generating function as well. 
    In the non-framed case, the moduli space of rank 2 stable sheaves on a surface has been shown to have product formula by Göttsche and Kool in \cite{GK_Rank2}. 
    A product formula for general rank would be extremely useful in extending the $\chi_y$ Vafa-Witten blowup formula of \cite{KLT_blowup} to elliptic genus, and similarly extending a forthcoming result of Arbesfeld, Kool and Laarakker \cite{AKL_Vertical}. 
\end{rem}

\subsection{Structure of the article}
The article is structured as follows:
\begin{enumerate}
    \item[-] 
    In Section \ref{sec: DT computation} we compute the partition function in terms of the refined topological vertex. The method is to combine Descombes' hyperbolic localisation formula \cite[Thm. 1.1]{Descombes_Hyperbolic} with the K-theoretic computational results of Arbesfeld \cite{Arbesfeld_K-Theoretic_DT}.
    \item[-] In Section \ref{sec: vertex calculations} we compute the formulas in Theorem \ref{thm: main} and Corollary \ref{cor: r=1 product} using Okounkov's vertex operators \cite{Okounkov_Infinite}. 
    We also obtain a trace formula for the elliptic genus of the moduli space of framed instanton sheaves on $\PP^2$ (given in Corollary \ref{cor: elliptic genus as trace final}). 
\end{enumerate}

\section{Computation of cohomological DT theory}
\label{sec: DT computation}

\subsection{Descombes' hyperbolic localisation formula}
\label{sec: DT computation hyperbolic}

    Cohomological DT invariants are very hard to compute, and the torus localisation techniques used in numerical DT theory \cite{GP_Localization,CKL_Torus} are insufficient in the cohomological setting. 
    It was suggested by Szendrői \cite[\S8.4]{Szendroi_Cohomological} that the ideas of Braden \cite{Braden_Hyperbolic} for hyperbolic localisation on intersection cohomological could be applied to cohomological DT theory as well. 
    Hyperbolic localisation for cohomological DT invariants was proved by Descombes in \cite{Descombes_Hyperbolic}, and will be the basis for the results of this article. 

    Hyperbolic localisation applies in the setting of an  algebraic space $X$ with the action of a one dimensional torus $\CC^*$. 
    The underlying concept of hyperbolic localisation is the decomposition of Białynicki-Birula \cite{Bialynicki-Birula_Some} (extended from smooth varieties by \cite{Braden_Hyperbolic,Drinfeld_Algebraic}).
    In this setting we consider the $\CC^*$-fixed space $X_0$ and also the attracting  variety $X^+$, which is the subset of points $x\in X$ such that $\lim_{t\rightarrow 0} t\cdot x$ exists.
    We then consider the hyperbolic localization diagram as in \cite{Braden_Hyperbolic}, \cite{Drinfeld_Algebraic} and
    \cite{Richarz_Spaces}:
    \begin{equation}
        \begin{tikzcd}
          X
            &
            X^{+}
            \arrow[l,swap,"\eta",hook']
            \arrow[r,"p"]
            &
            X^{\CC^*}
        \end{tikzcd}
        \label{eqn: hyperbolic localisation diagram}
    \end{equation}
    Here $\eta$ is the inclusion and $p$ is the natural morphism. 
    The functor $p_!\eta^*$ is called the \emph{hyperbolic localization functor}.

    To apply hyperbolic localisation in the setting of this article, we consider a Calabi-Yau threefold with a $\CC^*$-CY-action (i.e. locally the weights $(w_1,w_2,w_3)$ have $w_1+w_2+w_3=0$), and observe that this action extends to one on $\Hilb_{\beta,n}(Y)$.
    The $\CC^*$-action is compatible with the obstruction theory of \cite[Thm. 3.30]{Thomas_Holomorphic} and we obtain a \emph{$\CC^*$-equivariant perfect obstruction theory} $\phi: \calE^\bullet \rightarrow L^\bullet_{\Hilb_{\beta,n}(Y)}$.
    The $\CC^*$-fixed locus can then be decomposed into connected components
    \[
        \Hilb_{\beta,n}(Y)^{\CC^*} = \bigsqcup_{\lambda\in\Lambda} P_\lambda.
    \]
    
    In this article we focus on the case where each $P_\lambda$ is an \emph{isolated and reduced point}. 
    In this case
    \[
        \mathcal{H}^0(\calE^\bullet)^\vee \cong \Def_{P_\lambda},
    \]
    and we define the \text{index} at $P_\lambda$ to be 
    \[
        \Ind_{P_\lambda}
        :=
        \#\{ \mbox{Positve $\CC^*$-weights of $\Def_{P_\lambda}$}\}
        -
        \#\{ \mbox{Negative $\CC^*$-weights of $\Def_{P_\lambda}$}\}.
    \]
    Now, Descombes' hyperbolic localisation formula is given by the following theorem. 

    \begin{thm}
        [\mbox{\cite[Thm. 1.1]{Descombes_Hyperbolic}}]
        \label{thm: hyperbolic localisation}
        Suppose $\Hilb_{\beta,n}(Y)^{\CC^*}$ is a finite collection of reduced points.
        Then there is an isomorphism of monodromic mixed hodge modules
        \[
            p_!\eta^* \mathcal{DT}^{\bullet}_{Y,\beta,n}
            \cong
            \bigoplus_{\lambda \in \Lambda} 
            \LL^{\frac{1}{2} \Ind_{\lambda}}.
        \]
    \end{thm}

    \begin{rem}
        In \cite[Thm. 1.1]{Descombes_Hyperbolic}, the hyperbolic localisation formula is given in more generality. 
        In particular, the condition on the fixed loci is not needed. 
        In general, each fixed loci would have an associated monodromic mixed hodge module (of the form given in Theorem \ref{thm: MMHS existence}) appearing on the right hand side of the formula. 
        In the case of isolated smooth fixed points, we have used that this monodromic mixed hodge module is simply $\QQ$ as described in \cite[Eg. 2.11]{DM_Cohomological}. 
    \end{rem}

\subsection{Analysis of the torus fixed locus}

Analysis of the $\TT^2$-fixed subschemes of $X_{(1,1)}$ was carried out by Bryan in \cite[\S4.8]{Bryan_DonaldsonThomas} using toric techniques detailed in \cite{MNOP_I} and \cite[\S3 and App. B]{BCY_Orbifold}. 
Essentially, a torus-invariant subscheme is determined by combinatorial
data associated to its web diagram. The data involves a 2D partition for each edge and for each vertex, a 3D partition asymptotic to the the 2D partitions of the incident edges. 

The analysis of the $\TT^2$-fixed subschemes of $X_{(r,1)}$ is essentially the same as for $X_{(1,1)}$. 
The difference is seen by comparing \cite[Fig. 4]{Bryan_DonaldsonThomas} with the web diagram in Figure \ref{fig: fan and vertices}. Namely, the $\TT^2$-fixed locus of $X_{(r,1)}$ has $2r$ vertices and $3r$ edges. We then arrive at the following lemma.

\begin{lem}
    [\textbf{$\TT^2$-Fixed points}]
    \label{lem: T2 fixed points}
    For a curve class $\beta = \sum a_i\cdot A_i + b\cdot B + c\cdot C$ the fixed points of $\Hilb^{n,\beta}(Y)$ are indexed by the collection $\Lambda$ of tuples
    \[
        \lambda = (\bm{\alpha},\bm{\beta},\bm{\gamma}; \bm{\pi}, \bm{\eta})
    \]
    such that 
    \begin{enumerate}
        \item $\bm{\alpha} = (\alpha_0,\ldots,\alpha_{r-1})$ where $\alpha_i$ is a 2D partition with $|\alpha_i| = a_i$,
        \item $\bm{\beta} = (\beta_0,\ldots,\beta_{r-1})$ where $\beta_i$ is a 2D partition with $\sum |\beta_i| = b$,
        \item $\bm{\gamma} = (\gamma_0,\ldots,\gamma_{r-1})$ where $\gamma_i$ is a 2D partition with $\sum |\gamma_i| = c$,
        \item $\bm{\pi} = (\pi_0,\ldots,\pi_{r-1})$ where $\pi_i$ is a 3D-partition asymptotic to $(\alpha_i,\beta_i,\gamma_i)$,
        \item $\bm{\eta} = (\eta_0,\ldots,\eta_{r-1})$ where $\eta_i$ is a 3D-partition asymptotic to $(\alpha'_{i-1},\beta'_i,\gamma'_i)$,
    \end{enumerate}
    and 
    \begin{align}
        n = \frac{1}{2} \sum_i\Big( 2|\pi_i| + 2|\eta_i| + \| \alpha_i\|^2 + \| \alpha'_i\|^2 + \| \beta_i\|^2 + \| \beta'_i\|^2 + \| \gamma_i\|^2 + \| \gamma'_i\|^2  \Big).
        \label{eqn: fixed point condition}
    \end{align}
    Here for a 2D partition $\mu$ we have used $\| \mu \|^2 = \sum \mu_i^2$ and for a 3D-partition $\nu$ we have used the renormalised volume \cite{ORV_Quantum, BKY_Trace}:
    \[
        |\nu| := \sum_{\raisebox{-0.16em}{\mbox{\mancube}}\,\in\nu} \Big(1- \#\,\big\{\mbox{legs of $\nu$ containing \mancube}\big\} \Big).
    \]
\end{lem}

\begin{figure}[t]
    \raisebox{-0.5\height}{        
        \begin{tikzpicture}[
    scale=2.75,
    every node/.style={minimum size=1cm},
    on grid
    ] 
    % 1st hexagon
    \draw[A, ultra thick]     (0*1 + 0*0.707, 0*1 + 0*0.707) -- node[left,xshift=0.25cm] {\small$A_i$} ( 0*1 + 0*0.707, 1*1 + 0*0.707);
    \draw[C, ultra thick]   ( 0*1 + 0*0.707, 0*1 + 0*0.707) -- node[below,yshift=0.25cm] {\small$C_i$} ( 1*1 + 0*0.707, 0*1 + 0*0.707);
    \draw[B, ultra thick]    ( 0*1 + 0*0.707, 1*1 + 0*0.707) -- node[above,yshift=-0.3cm,xshift=-0.15cm] {\small$B_{i+1}$} ( 0*1 + 1*0.707, 1*1 + 1*0.707);
    \draw[B, ultra thick]    ( 1*1 + 0*0.707, 0*1 + 0*0.707) -- node[below,yshift=0.3cm,xshift=0.15cm] {\small$B_i$} ( 1*1 + 1*0.707, 0*1 + 1*0.707);
    \draw[A, ultra thick]     ( 1*1 + 1*0.707, 0*1 + 1*0.707) -- node[right,xshift=-0.25cm] {\small$A_i$} ( 1*1 + 1*0.707, 1*1 + 1*0.707);
    \draw[C, ultra thick]   ( 0*1 + 1*0.707, 1*1 + 1*0.707) -- node[above,yshift=-0.2cm] {\small$C_{i+1}$} ( 1*1 + 1*0.707, 1*1 + 1*0.707);
    %
    % 1st Half-legs
    \draw[B, ultra thick]     ( 0*1 + 0*0.707, 0*1 + 0*0.707) -- node[left,yshift=-0.15cm,xshift=-0.1cm] {\small$B_i$} ( 0*1 - 0.5*0.707, 0*1 - 0.5*0.707);
    \draw[C, ultra thick]   ( 0*1 + 0*0.707, 1*1 + 0*0.707) -- node[above,yshift=-0.2cm,xshift=-0.2cm] {\small$C_{i+1}$} ( -0.5*1 + 0*0.707, 1*1 + 0*0.707);
    \draw[C, ultra thick]   ( 1*1 + 1*0.707, 0*1 + 1*0.707) -- node[below,yshift=0.2cm,xshift=0.2cm] {\small$C_i$} ( 1.5*1 + 1*0.707, 0*1 + 1*0.707);
    \draw[A, ultra thick]     ( 1*1 + 0*0.707, 0*1 + 0*0.707) -- node[left,yshift=-0.07cm,xshift=0.1cm] {\small$A_{i-1}$} ( 1*1 + 0*0.707, -0.5*1 + 0*0.707);
    %
    % 3rd Half-legs
    \draw[A, ultra thick]     ( 0*1 + 1*0.707, 1*1 + 1*0.707) -- node[right,yshift=0.1cm,xshift=-0.15cm] {\small$A_{i+1}$} ( 0*1 + 1*0.707, 1.5*1 + 1*0.707);
    \draw[B, ultra thick]     ( 1*1 + 1*0.707, 1*1 + 1*0.707) -- node[right,yshift=0.17cm,xshift=-0.0cm] {\small$B_i$} ( 1*1 + 1.5*0.707, 1*1 + 1.5*0.707);
    %
    % Vertices
    \node[fill,circle,minimum size=3.2pt,inner sep=0,] at ( 0*1 + 0*0.707, 0*1 + 0*0.707) {};
    \node[fill,circle,minimum size=3.2pt,inner sep=0,] at ( 0*1 + 0*0.707, 1*1 + 0*0.707) {};
    \node[fill,circle,minimum size=3.2pt,inner sep=0,] at ( 1*1 + 0*0.707, 0*1 + 0*0.707) {};
    \node[fill,circle,minimum size=3.2pt,inner sep=0,] at ( 0*1 + 1*0.707, 1*1 + 1*0.707) {};
    \node[fill,circle,minimum size=3.2pt,inner sep=0,] at ( 1*1 + 1*0.707, 0*1 + 1*0.707) {};
    \node[fill,circle,minimum size=3.2pt,inner sep=0,] at ( 1*1 + 1*0.707, 1*1 + 1*0.707) {};
    %
    % WEIGHTS 1st hexagon
    \draw[-stealth,very thick]     (0.1*1 + 0*0.707, 0.1*1 + 0*0.707) -- ( 0.4*1 + 0*0.707, 0.1*1 + 0*0.707)
    node[above,yshift=-0.25cm] {\small$w_2$}
    ;
    \draw[-stealth,very thick]   ( 0.1*1 + 0*0.707, 0.1*1 + 0*0.707) -- ( 0.1*1 + 0*0.707, 0.4*1 + 0*0.707)
    node[right,xshift=-0.25cm] {\small$w_1$}
    ;
    \node[fill,circle,minimum size=1.6pt,inner sep=0,] at (0.1*1 + 0*0.707, 0.1*1 + 0*0.707) {}
    ;
    \draw[-stealth,very thick]   ( 0.95*1 + 0*0.707, 0.1*1 + 0*0.707) -- ( 0.65*1 + 0*0.707, 0.1*1 + 0*0.707)
    node[above,yshift=-0.25cm] {\small$-w_2$}
    ;
    \draw[-stealth,very thick]    ( 0.95*1 + 0*0.707, 0.1*1 + 0*0.707) -- ( 0.95*1 + 0.3*0.707, 0.1*1 + 0.3*0.707)
    node[above,yshift=-0.3cm,xshift=-0.3cm] {\small$w_2+w_1$}
    ;
    \node[fill,circle,minimum size=1.6pt,inner sep=0,] at ( 0.95*1 + 0*0.707, 0.1*1 + 0*0.707) {};
    \draw[-stealth,very thick]    ( 0.9*1 + 1*0.707, 0.05*1 + 1*0.707) -- ( 0.9*1 + 0.7*0.707, 0.05*1 + 0.7*0.707)
    node[left,yshift=0.1cm,xshift=0.0cm] {\small$-w_2-w_1$}
    ;
    \draw[-stealth,very thick]     ( 0.9*1 + 1*0.707, 0.05*1 + 1*0.707) -- ( 0.9*1 + 1*0.707, 0.35*1 + 1*0.707)
    node[left,xshift=0.25cm] {\small$w_1$}
    ;
    \node[fill,circle,minimum size=1.6pt,inner sep=0,] at ( 0.9*1 + 1*0.707, 0.05*1 + 1*0.707) {};
    \draw[-stealth,very thick]   ( 0.9*1 + 1*0.707, 0.9*1 + 1*0.707) -- ( 0.6*1 + 1*0.707, 0.9*1 + 1*0.707)
    node[below,yshift=0.25cm] {\small$-w_2$}
    ;
    \draw[-stealth,very thick]   ( 0.9*1 + 1*0.707, 0.9*1 + 1*0.707) -- ( 0.9*1 + 1*0.707, 0.6*1 + 1*0.707)
    node[left,xshift=0.1cm] {\small$-w_1$}
    ;
    \node[fill,circle,minimum size=1.6pt,inner sep=0,] at ( 0.9*1 + 1*0.707, 0.9*1 + 1*0.707) {};
    \draw[-stealth,very thick]   ( 0.05*1 + 1*0.707, 0.9*1 + 1*0.707) -- ( 0.35*1 + 1*0.707, 0.9*1 + 1*0.707)
    node[below,yshift=0.25cm] {\small$w_2$}
    ;
    \draw[-stealth,very thick]   ( 0.05*1 + 1*0.707, 0.9*1 + 1*0.707) -- ( 0.05*1 + 0.7*0.707, 0.9*1 + 0.7*0.707)
    node[right,xshift=-0.2cm,yshift=-0.25cm] {\small$-w_2-w_1$}
    ;
    \node[fill,circle,minimum size=1.6pt,inner sep=0,] at ( 0.05*1 + 1*0.707, 0.9*1 + 1*0.707) {};
    \draw[-stealth,very thick]   ( 0.1*1 + 0*0.707, 0.95*1 + 0*0.707) -- ( 0.1*1 + 0.3*0.707, 0.95*1 + 0.3*0.707)
    node[right,xshift=-0.1cm,yshift=-0.2cm] {\small$w_2+w_1$} 
    ;
    \draw[-stealth,very thick]   ( 0.1*1 + 0*0.707, 0.95*1 + 0*0.707) -- ( 0.1*1 + 0*0.707, 0.65*1 + 0*0.707)
    node[right,xshift=-0.1cm] {\small$-w_1$}
    ;
    \node[fill,circle,minimum size=1.6pt,inner sep=0,] at ( 0.1*1 + 0*0.707, 0.95*1 + 0*0.707) {};
    %
    % WEIGHTS half edges
    \draw[-stealth,very thick]   ( 0.05*1 + 0*0.707, -0.1*1 + 0*0.707) -- node[below,xshift=0.15cm,yshift=0.0cm] {\small$-w_2-w_1$} ( 0.05*1  -0.3*0.707, -0.1*1 - 0.3*0.707);
    \node[fill,circle,minimum size=1.6pt,inner sep=0,] at ( 0.05*1 + 0*0.707, -0.1*1 + 0*0.707) {};
    \draw[-stealth,very thick]   ( 1.1*1 + 0*0.707, -0.05*1 + 0*0.707) -- ( 1.1*1 + 0*0.707, -0.35*1 + 0*0.707)
    node[right,xshift=-0.1cm] {\small$-w_1$} 
    ;
    \node[fill,circle,minimum size=1.6pt,inner sep=0,] at ( 1.1*1 + 0*0.707, -0.05*1 + 0*0.707) {};
    \draw[-stealth,very thick]   ( 1.1*1 + 1*0.707, 0.1*1 + 1*0.707) -- ( 1.4*1 + 1*0.707, 0.1*1 + 1*0.707)
    node[above,yshift=-0.25cm] {\small$w_2$}
    ;
    \node[fill,circle,minimum size=1.6pt,inner sep=0,] at ( 1.1*1 + 1*0.707, 0.1*1 + 1*0.707) {};
    \draw[-stealth,very thick]   ( 0.95*1 + 1*0.707, 1.1*1 + 1*0.707) -- ( 0.95*1 + 1.3*0.707, 1.1*1 + 1.3*0.707)
    node[left,xshift=-0.1cm,yshift=-0.0cm] {\small$w_2+w_1$} 
    ;
    \node[fill,circle,minimum size=1.6pt,inner sep=0,] at ( 0.95*1 + 1*0.707, 1.1*1 + 1*0.707) {};
    \draw[-stealth,very thick]   ( -0.1*1 + 1*0.707, 1.05*1 + 1*0.707) -- ( -0.1*1 + 1*0.707, 1.35*1 + 1*0.707)
    node[left,xshift=0.25cm] {\small$w_1$}
    ;
    \node[fill,circle,minimum size=1.6pt,inner sep=0,] at ( -0.1*1 + 1*0.707, 1.05*1 + 1*0.707) {};
    \draw[-stealth,very thick]   ( -0.1*1 + 0*0.707, 0.9*1 + 0*0.707) -- ( -0.4*1 + 0*0.707, 0.9*1 + 0*0.707)
    node[below,yshift=0.25cm] {\small$-w_2$}
    ;
    \node[fill,circle,minimum size=1.6pt,inner sep=0,] at ( -0.1*1 + 0*0.707, 0.9*1 + 0*0.707) {};
\end{tikzpicture}
    }
    \caption{A depiction of how the Calabi-Torus $\TT^2$ acts on $X_{(r,1)}$.}
    \label{fig: T^2 action}
\end{figure}
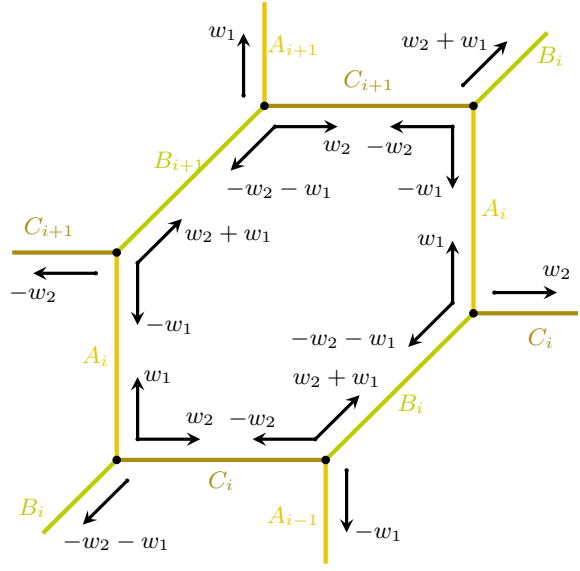

To apply the hyperbolic localisation formula, we may choose a $\CC^*$-subtorus of $\TT^2$ given by weights $(w_1,w_w)$ and with character denoted by $t$. 
The local description of this is given in Figure \ref{fig: T^2 action}. 
The (general) hyperbolic localisation formula is unaffected by the choice of $\CC^*$-subtorus and we choose a subtorus such that 
\[
    \Hilb^{\beta,n}\big(X_{(r,1)}\big)^{\CC^*} 
    =
    \Hilb^{\beta,n}\big(X_{(r,1)}\big)^{\TT^2}
\]
so that we can apply Theorem \ref{thm: hyperbolic localisation}.
(The condition required is that $w_1 \neq n\cdot w_2$ for any $n\in \ZZ$.) 
We will use the results of \cite[\S4]{Arbesfeld_K-Theoretic_DT}, where a choice of $\CC^*$-subtorus is called a \textit{slope}. 
Indeed, we assume for later convenience that $w_1>>w_2>0$ in which case we are using a \emph{preferred slope} of \cite{Arbesfeld_K-Theoretic_DT}.

\begin{prop}
    \label{prop: attractor locus}
    $Hilb^{\beta,n}(X_{(r,1)})$ is equal to its attracting locus, $Hilb^{\beta,n}(X_{(r,1)})^{+}$.
\end{prop}
\begin{proof}
    This follows because the web diagram of $X_{(r,1)}$ contains no half edges, meaning that for $[Z]\in \Hilb^{\beta,n}(X_{(r,1)})$ we have that $\lim_{t\rightarrow 0} t\cdot [Z]$ always exists. 
\end{proof}

The obstruction theory and the associated $\CC^*$ weights can be computed from the above combinatorial data following the method of \cite[\S4.7-\S4.9]{MNOP_I} (see also, \cite[\S8.2]{NO_Membranes}, \cite[\S4]{Arbesfeld_K-Theoretic_DT}). 
In particular from \cite[\S4.9]{MNOP_I}, for a 2D partition $\mu$ and 3D-partition $\nu$ there are Laurent polynomials $E_{\mu}(w_1,w_2,w_3)$ and $V_{\nu}(w_1,w_2,w_3)$ called \emph{edge} and \emph{vertex functions} such that 
\begin{align}
    T^{\vir}|_{\lambda}
    =
    \sum_{i}
    &
    E_{\alpha_i}(w_1,-w_1-w_2,w_2)
    + E_{\beta_i}(w_1,-w_1-w_2,w_2)
    + E_{\gamma_i}(w_1,-w_1-w_2,w_2)
    \nonumber
    \\
    &
    + V_{\pi_i}(w_1,-w_1-w_2,w_2)
    + V_{\eta_i}(-w_1,w_1+w_2,-w_2)
    \label{eqn: T^vir decomposition}
\end{align}
where $T^{\vir} = (\calE^0)^\vee - (\calE^{-1})^\vee \in K_0\big(\Hilb^{\beta,n}\big(X_{(r,1)}\big)\big)$ is the virtual tangent bundle.
Here we have used the conventions from \cite[(4.8) and (4.9)]{Arbesfeld_K-Theoretic_DT} so that the arguments can be read directly off Figure \ref{fig: T^2 action}.

\begin{defn}
    \label{def: index of edge and vertex}
    We define the \emph{index associated to the edge and vertex functions} via the following procedure. 
    First we recall from \cite[(4.11)]{Arbesfeld_K-Theoretic_DT} that 
    \[
        E_{\mu}(\_) = - E_{\mu}(\_)^\vee
        \hspace{1em}
        \mbox{and}
        \hspace{1em}
        V_{\nu}(\_) = -V_{\nu}(\_)^\vee
    \]
    so we have decompositions
    \begin{align}
        E_{\mu}(\_) 
        = \sum_j \left(u_j - \frac{1}{u_j}\right)
        \hspace{1em}
        \mbox{and}
        \hspace{1em}
        V_{\nu}(\_) 
        = \sum_j \left(v_j - \frac{1}{v_j}\right).
        \label{eqn: ind decomposition choice}
    \end{align}
    Then we write
    \[
        \sum u_i = f_{-}\big(\tfrac{1}{t}\big) + a_0 + f_{+}(t)
        \hspace{1em}
        \mbox{and}
        \hspace{1em}
        \sum v_i = g_{-}\big(\tfrac{1}{t}\big) + b_0 + g_{+}(t) 
    \]
    for $\ZZ$-polynomials $f_{\pm}$ and $g_{\pm}$ with no constant term and $a_0,b_0\in\ZZ$.
    We can then define
    \[
        \Ind\big(E_{\mu}(\_)\big):= f_{+}(1) - f_{-}(1)
        \hspace{1em}
        \mbox{and}
        \hspace{1em}
        \Ind\big(V_{\nu}(\_)\big):= g_{+}(1) - g_{-}(1),
    \]
    noting that these are independent of the choice from \eqref{eqn: ind decomposition choice}.
\end{defn}

The data of $\Ind_\lambda$ is contained in $T^{\vir}|_{\lambda}$ and can be computed via the edge and vertex index functions. 
Recall from \cite{Thomas_Holomorphic} that $\Def_\lambda = \Obs_\lambda^\vee$, so we have 
$
    T^{\vir}|_{\lambda} 
    = \Def_\lambda - \Def_\lambda^\vee.
$
Thus, extending $\Ind$ additively to \eqref{eqn: T^vir decomposition} gives
$
    \Ind(T^{\vir}|_{\lambda}) = \Ind_\lambda.
$

We now follow \cite[\S4.2]{Arbesfeld_K-Theoretic_DT} and define 
\begin{align*}
    \mathsf{E}_{\lambda}(w_1,w_2,w_3)
    &
    :=
    q^{\frac{1}{2}(\|\lambda\|^2 +\|\lambda'\|^2)} \LL^{\frac{1}{2}\Ind(\mathrm{E}_{\lambda}(w_1,w_2,w_3))},
    \\
    \mathsf{V}_{\alpha,\beta,\gamma}(w_1,w_2,w_3)
    &
    :=
    \sum_{\pi \in P(\alpha,\beta,\gamma)} 
    q^{|\lambda|} \LL^{\frac{1}{2}\Ind(\mathrm{V}_{\pi}(w_1,w_2,w_3))}
\end{align*}
where
\[
    P(\alpha,\beta,\gamma)
    :=
    \big\{~
        \mbox{3D partitions with legs asymptotic to $\alpha,\beta,\gamma$}
    ~\big\}.
\]
(Note: These correspond to $\hat{E}_{\lambda,(-1,-1)}^{\sigma}$ and $\hat{V}_{\alpha,\beta,\gamma}^{\sigma}$ from \cite[\S4.3]{Arbesfeld_K-Theoretic_DT}.)

Hence using Theorem \ref{thm: hyperbolic localisation} (hyperbolic localisation) and Proposition \ref{prop: attractor locus} we arrive at the following Lemma. 

\begin{lem}
    \label{lem: Z_DT partial in terms of E and V}
    If $\beta = \sum a_i\cdot A_i + b\cdot B + c\cdot C$ then
    \[
        \sum_n \mathbb{DT}_{\beta,n}(X_{(r,1)}) \, q^n
    \]
    is given by the following expression in $K(\MMHS)[q^{-1},q\psr$:
    \begin{align*}
        \def\arraystretch{1.5}
        \sum_{(\bm{\alpha},\bm{\beta},\bm{\gamma})}
        \prod_i
        \!
        \left(
            \!\!\!
            \begin{array}{l}
                \mathsf{E}_{\alpha_i}(-w_1-w_2,w_2,w_1)
                \cdot \mathsf{E}_{\beta_i}(w_2,w_1,-w_1-w_2)
                \cdot \mathsf{E}_{\gamma_i}(w_1,-w_1-w_2,w_2)
                \\
                ~
                \cdot\, \mathsf{V}_{\alpha_i,\beta_i,\gamma_i}(w_1,-w_1-w_2,w_2)
                \cdot \mathsf{V}_{\alpha'_i,\beta'_i,\gamma'_{i-1}}(-w_1,w_1+w_2,-w_2)
            \end{array}
            \!\!\!\!
        \right)
    \end{align*}  
    where the sum is over all tuples $(\bm{\alpha},\bm{\beta},\bm{\gamma})$ as described in Lemma \ref{lem: T2 fixed points} without the conditions \emph{(iv)}, \emph{(v)}, and \eqref{eqn: fixed point condition}.
\end{lem}

Following \cite[\S4.3.2 and \S4.3.3]{Arbesfeld_K-Theoretic_DT} we obtain explicit formulas for $\mathsf{E}_{\mu}(\_)$ and $\mathsf{V}_{\alpha,\beta,\gamma}(\_)$. 
First for convenience we define 
\begin{align}
    u := -q\, \LL^{1/2}
    \hspace{1em}
    \mbox{and}
    \hspace{1em}
    v := -q\, \LL^{-1/2},
    \label{eqn: u and v definition}
\end{align}
and recall the \emph{refined topological vertex} of Iqbal, Kozçaz and Vafa \cite{IKV_Refined} (modified following \cite[\S4.3.3]{Arbesfeld_K-Theoretic_DT})
\[
    \mathsf{C}_{\alpha, \beta, \gamma}(u,v)
    :=
    \frac{
        M(v^{-1};u,v)\,
        u^{-\frac{1}{2}\|\alpha'\|^2}
        v^{-\frac{1}{2}\|\beta\|^2}
    }
    {
        \prod\limits_{x\in \gamma} \left(1 - u^{a_{\gamma}(x)+1} v^{l_{\gamma}(x)} \right)
    }
    \sum\limits_{\eta}
        \left(\frac{v}{u}\right)^{\frac{1}{2}|\eta|}
        S_{\alpha/\eta}( u^{-\rho } v^{-\gamma'}) 
        S_{\beta'/\eta}(u^{-\gamma} v^{-\rho } )
\]
where $\rho:= (-n+\frac{1}{2})_{n\in\NN}$.
Now, applying \cite[Prop. 4.2 and Prop 4.6]{Arbesfeld_K-Theoretic_DT}\footnote{
    Each row of \cite[Prop. 4.2]{Arbesfeld_K-Theoretic_DT} has a typo: $t$ should be interchanged with $q$. 
    To see this, observe that on page 27 of the unabridged version of \cite{Arbesfeld_K-Theoretic_DT} (\href{https://arxiv.org/abs/1905.04567}{arXiv:1905.04567v2}), the contribution of $\square \in \lambda$ is stated as $l(\square)- a(\square)$, but is actually $a(\square)- l(\square)$. 
}
we have
\begin{align*}
    \mathsf{E}_{\alpha}(w_1,-w_1-w_2,w_2)
    &= 
    u^{\frac{1}{2}\|\alpha'\|^2}v^{\frac{1}{2}\|\alpha\|^2}
    \\
    \mathsf{E}_{\beta}(-w_1-w_2,w_2,w_1)
    &= 
    u^{\frac{1}{2}\|\beta'\|^2}v^{\frac{1}{2}\|\beta\|^2}
    \\
    \mathsf{E}_{\gamma}(w_2,w_1,-w_1-w_2)
    &= 
    u^{\frac{1}{2}\|\gamma\|^2}v^{\frac{1}{2}\|\gamma'\|^2}
    \\
    \mathsf{V}_{\alpha,\beta,\gamma}(w_1,-w_1-w_2,w_2)
    &= 
    \mathsf{C}_{\alpha, \beta, \gamma}(u,v)
    \\
    \mathsf{V}_{\alpha',\beta',\gamma'}(-w_1,w_1+w_2,-w_2)
    &= 
    \mathsf{C}_{\alpha', \beta', \gamma'}(v,u)
\end{align*}

Combining with Lemma \ref{lem: Z_DT partial in terms of E and V} we obtain
\begin{cor}
    \label{cor: DT sum as refined vertex}
    If $\beta = \sum a_i\cdot A_i + b\cdot B + c\cdot C$ then we have the following equality in $K(\MMHS)[q^{-1},q\psr$:
    \begin{align*}
        \sum_n \mathbb{DT}_{\beta,n}(X_{(r,1)})\, q^n
        =
        \def\arraystretch{1.5}
        \sum_{(\bm{\alpha},\bm{\beta},\bm{\gamma})}
        \prod_i
        \left(
            \begin{array}{l}
                u^{\frac{1}{2}(\|\alpha'_i\|^2+\|\beta'_i\|^2+\|\gamma_i\|^2)}
                v^{\frac{1}{2}(\|\alpha_i\|^2+\|\beta_i\|^2+\|\gamma'_i\|^2)}
                \\
                ~
                \cdot\,
                \mathsf{C}_{\alpha_i, \beta_i, \gamma_i}(u,v)
                \cdot
                \mathsf{C}_{\alpha'_{i-1}, \beta'_i, \gamma'_i}(v,u)
            \end{array}
        \right)
    \end{align*} 
    where the sum is over all tuples $(\bm{\alpha},\bm{\beta},\bm{\gamma})$ as described in Lemma \ref{lem: T2 fixed points} without the conditions \emph{(iv)}, \emph{(v)}, and \eqref{eqn: fixed point condition}.
\end{cor}

\section{Vertex calculations}
\label{sec: vertex calculations}

\subsection{Trace expression for cohomological partition function}

We begin this section by considering a product factor which arises from the refined topological vertex.
That is, for a 2D partition $\gamma$ define
\begin{align*}
    Y_\gamma
    :=
    u^{\frac{1}{2}\|\gamma\|^2}v^{\frac{1}{2}\|\gamma'\|^2}
    \frac{
        M(v^{-1};u,v)\,
    }
    {
        \prod\limits_{x\in \gamma} \left(1 - u^{a_{\gamma}(x)+1} v^{l_{\gamma}(x)} \right)
    }
    \frac{
        M(u^{-1};u,v)\,
    }
    {
        \prod\limits_{x\in \gamma'} \left(1 -  u^{l_{\gamma'}(x)} v^{a_{\gamma'}(x)+1} \right)
    }.
\end{align*}

\begin{lem}
    \label{lem: infinite to finite sum}
    Let $\mu$ and $\eta$ be 2D partitions. Then we have the following equation in $\ZZ[u^{-1},v^{-1}]\psl u,v\psr$: 
    \begin{align*}
        &
        \sum_{x\in \mu}  u^{-a_{\mu}(x)-1}v^{-l_{\eta}(x)}
        +
        \sum_{x\in \eta}  u^{a_{\eta}(x)}v^{l_{\mu}(x)+1}
        \\
        &=
        \sum_{x\in \mathbb{N}^2}  u^{-a_{\mu}(x)-1}v^{-l_{\eta}(x)}
        -
        \sum_{x\in \mathbb{N}^2}  u^{-a_{\emptyset}(x)-1}v^{-l_{\emptyset}(x)}
    \end{align*}
\end{lem}
\begin{proof}
    Nakajima and Yoshioka show in the proof of \cite[Thm. 2.11]{NY_Instanton_I} that the top line is equal to
    \[
        \sum_{i=1}^{\eta'_1} \sum_{j=1}^{\mu_1}
        \left(
            u^{-a_{\mu}(i,j)-1} v^{-l_{\eta}(i,j)} 
            -
            u^{-a_{\emptyset}(i,j)-1} v^{-l_{\emptyset}(i,j)} 
        \right)
        +
        \sum_{(i,j)\in \mu'}
        u^{\eta'_1 - i}v^{j} 
        +
        \sum_{(i,j)\in \eta}
        u^{i-1} v^{\mu_1-j+1} 
    \]
    Observe that for $i > \eta'_1$ we have $l_{\eta}(i,j) = - j$. Hence we have
    \begin{align*}
        \sum_{i = \eta'_1+1}^\infty \sum_{j=1}^{\mu_1}
        \left(
            u^{-a_{\mu}(i,j)-1}  v^{-l_{\eta}(i,j)} 
            -
            u^{-a_{\emptyset}(i,j)-1} v^{-l_{\emptyset}(i,j)}
        \right)
        &
        =
        \sum_{j=1}^{\mu_1}
        v^{j} 
        \sum_{i = \eta'_1+1}^\infty 
        \left(
            u^{-\mu'_j+i-1}
            -
            u^{i-1}
        \right)
        \\
        &
        =
        \sum_{j=1}^{\mu_1}
        v^{j}
        \, 
        \sum_{k = 1}^{ \mu'_j}
            u^{\eta'_1-k}. 
    \end{align*}
    The term for $j > \mu_1$, $i \leq \eta'_1$ is similarly described, and the term for $i > \eta'_1$, $j > \mu_1$ is zero.  
    The desired result now follows immediately.
\end{proof}

\begin{cor}
    \label{cor: infinite to finite product}
    Using $M(k; u, v) := \prod_{n,m > 0} ( 1 - k\, u^n v^m)^{-1} $
    we have the equality
    \begin{align*}
        &
        M(k\,u^{-1}; u,v)\cdot
        \prod_{n,m>0} 
        \Big(1 - k\,  u^{-a_{\mu}(x)-1}v^{-l_{\eta}(x)}\Big)
        \\
        &
        =
        \prod_{x\in \mu} 
        \Big(1 - k\, u^{-a_{\mu}(x)-1}v^{-l_{\eta}(x)}\Big)
        \prod_{z\in \eta} 
        \Big(1 - k\, u^{a_{\eta}(x)}v^{l_{\mu}(x)+1}\Big)
        .
    \end{align*}
\end{cor}

Now,  we apply Corollary \ref{cor: infinite to finite product} after recalling that
\begin{enumerate}[itemsep=0.5em]
    \item 
    from the definition $a_{\gamma}(i,j) = l_{\gamma'}(j,i)$ as well as $l_{\gamma}(i,j) = a_{\gamma'}(j,i)$
    \item 
    $\sum_{x\in \mu} a_{\mu}(x) = \frac{1}{2}(\|\mu\|^2-|\mu|)$
    and
    $\sum_{x\in \mu} (l_{\mu}(x) +1)
    =
    \frac{1}{2}(\|\mu'\|^2+|\mu'|)$,
\end{enumerate}
to obtain the result:

\begin{align*}
    Y_\gamma
    &
    =
    (-1)^{|\gamma|} \left( \frac{v}{u}\right)^{\frac{1}{2}|\gamma|}
    \frac{
        M(v^{-1};u,v)\,
    }
    {
        \prod\limits_{x\in \gamma} \left(1 - u^{-a_{\gamma}(x)-1} v^{-l_{\gamma}(x)} \right)
    }
    \frac{
        M(u^{-1};u,v)\,
    }
    {
        \prod\limits_{x\in \gamma} \left(1 -  u^{a_{\gamma}(x)} v^{l_{\gamma}(x)+1} \right)
    }
    \\
    &
    =
    (-1)^{|\gamma|} \left( \frac{v}{u}\right)^{\frac{1}{2}|\gamma|}
    \frac{
        M(v^{-1};u,v)\,
    }
    {
        \prod\limits_{x\in \gamma} \left(1 - u^{-a_{\gamma}(x)-1} v^{-l_{\gamma}(x)} \right)
    }.
\end{align*}

Now, for an $r$-tuple of 2D partitions $\bm{\gamma} = (\gamma_0,\ldots,\gamma_{r-1})$, define
\begin{align*}
    Z_{\bm{\gamma}} :=
    &
    \sum_{(\bm{\alpha},\bm{\beta})}
    \prod_i
    ~
    Q_{A_i}^{|\alpha_i|} \, Q_{B}^{|\beta_i|}
    ~
    \sum\limits_{\eta_i}
        \left(\frac{v}{u}\right)^{\frac{1}{2}|\eta_i|}
        S_{\alpha_i/\eta_i}( u^{-\rho } v^{-\gamma_i'}) 
        S_{\beta_i'/\eta_i}(u^{-\gamma_i} v^{-\rho} )
    \\
    &
    \hspace{8.3em}
    \cdot
    \sum\limits_{\delta_i}
        \left(\frac{u}{v}\right)^{\frac{1}{2}|\delta_i|}
        S_{\beta_i/\delta_i}( u^{-\rho }  v^{-\gamma'_i})
        S_{\alpha'_{i-1}/\delta_i}( u^{-\gamma_i} v^{-\rho } ).
\end{align*} 

So that from Corollary \ref{cor: DT sum as refined vertex} we have
\begin{align}
    Z_c^{\mathbb{DT}}
    =
    &
    \multisum{\mbox{\scriptsize$\bm{\gamma}$ with} \\ |\bm{\gamma}|= c}
    \sum_{(\bm{\alpha},\bm{\beta})}
    \prod_i
    \left(
        \def\arraystretch{1.5}
        \begin{array}{l}
            u^{\frac{1}{2}(\|\alpha'_i\|^2+\|\beta'_i\|^2+\|\gamma_i\|^2)}
            v^{\frac{1}{2}(\|\alpha_i\|^2+\|\beta_i\|^2+\|\gamma'_i\|^2)}
            \\
            ~
            \cdot\,
            \mathsf{C}_{\alpha_i, \beta_i, \gamma_i}(u,v)
            \cdot
            \mathsf{C}_{\alpha'_{i-1}, \beta'_i, \gamma'_i}(v,u)
        \end{array}
    \right)
    \nonumber
    \\
    =
    &
    \multisum{\mbox{\scriptsize$\bm{\gamma}$ with} \\ |\bm{\gamma}|= c}
    Y_{\gamma_0}\cdots Y_{\gamma_{r-1}}\cdot Z_{\bm{\gamma}}.
    \label{eqn: Z^DT Y-Z split}
\end{align}

We will now determine an operator trace expression for $Z_{\bm{\gamma}}$.
Defining the change of variables
$\Qalt_{A_i} := Q_{A_i} \LL^{-\frac{1}{2}} = Q_{A_i} u^{-1/2} v^{1/2} $
and
$\Qalt_{B} := Q_{B} \LL^{\frac{1}{2}} = Q_{B} u^{1/2} v^{-1/2} $
we have
\begin{align*}
    Z_{\bm{\gamma}} 
    :=
    &
    \sum_{(\bm{\alpha},\bm{\beta})}
    \prod_i
    ~
    \Qalt_{A_i}^{|\alpha_i|} 
    \,
    \Qalt_{B}^{|\beta_i|}
    ~
    \sum\limits_{\eta_i}
        \left(\frac{v}{u}\right)^{\frac{1}{2}|\eta_i|}
        u^{\frac{1}{2}|\alpha_{i}|}
        v^{-|\alpha_i|+\frac{1}{2}|\beta_{i}|}
        S_{\alpha_i/\eta_i}( u^{-\rho } v^{-\gamma_i'}) 
        S_{\beta_i'/\eta_i}(u^{-\gamma_i} v^{-\rho} )
    \\
    &
    \hspace{8.3em}
    \cdot
    \sum\limits_{\delta_i}
        \left(\frac{u}{v}\right)^{\frac{1}{2}|\delta_i|}
        u^{-\frac{1}{2}|\beta_{i}|}
        v^{\frac{1}{2}|\alpha_{i-1}|}
        S_{\beta_i/\delta_i}( u^{-\rho }  v^{-\gamma'_i})
        S_{\alpha'_{i-1}/\delta_i}( u^{-\gamma_i} v^{-\rho } )
    \\
    =
    &
    \sum_{(\bm{\alpha},\bm{\beta})}
    \prod_i
    ~
    \Qalt_{A_i}^{|\alpha_i|} 
    \,
    \Qalt_{B}^{|\beta_i|}
    ~
    \sum\limits_{\eta_i}
        \left(\frac{v}{u}\right)^{\frac{1}{2}|\eta_i|}
        S_{\alpha_i/\eta_i}( u^{-\rho +\frac{1}{2} } v^{-\gamma_i'-1}) 
        S_{\beta_i'/\eta_i}(u^{-\gamma_i} v^{-\rho+\frac{1}{2}} )
    \\
    &
    \hspace{8.3em}
    \cdot
    \sum\limits_{\delta_i}
        \left(\frac{u}{v}\right)^{\frac{1}{2}|\delta_i|}
        S_{\beta_i/\delta_i}( u^{-\rho-\frac{1}{2} }  v^{-\gamma'_i})
        S_{\alpha'_{i-1}/\delta_i}( u^{-\gamma_i} v^{-\rho+\frac{1}{2} } ).
\end{align*}

We can rewrite the above expression using vertex operators of \cite{Okounkov_Infinite} (see also \cite{OR_Random,Young_Generating,BKY_Trace}). 
These operators have the properties
\[
    \Gamma_{-}(\mathbf{x}) \,\big|\, \mu\big> = \sum_{\lambda \succ \mu} S_{\lambda/\mu}(\mathbf{x})  
    \,\big|\, \lambda\big>,
    \hspace{1.5em}
    %\mbox{and }
    \hspace{1.5em}
    \Gamma_{+}(\mathbf{x}) \,\big|\, \lambda\big> = \sum_{ \mu \prec \lambda} S_{\lambda/\mu}(\mathbf{x}) 
    \,\big|\, \mu\big>,
\]
\[
    \Gamma'_{-}(\mathbf{x}) \,\big|\, \mu\big> = \sum_{\lambda' \succ \mu'} S_{\lambda'/\mu'}(\mathbf{x})  
    \,\big|\, \lambda\big>,
    \hspace{1.5em}
    %\mbox{and }
    \hspace{1.5em}
    \Gamma'_{+}(\mathbf{x}) \,\big|\, \lambda\big> = \sum_{ \mu' \prec \lambda'} S_{\lambda'/\mu'}(\mathbf{x}) 
    \,\big|\, \mu\big>, 
    \hspace{1.5em}
    \mbox{and }
\]
\[
    Q^H \,\big|\, \mu\big> 
    =
    Q^{|\mu|}\,\big|\, \mu\big>.
\]
Hence, we have the following expression:
\begin{align*}
    Z_{\bm{\gamma}}
    &
    =
    \mathrm{tr}\!\left(
        \prod_{i=0}^{r-1} 
        \,\Qalt_{A_i}^H\,
        \Gamma'_{-}\Big(u^{-\gamma_i} v^{-\rho\textcolor{black}{+\frac{1}{2}}}\Big) 
        \Gamma'_{+}\Big( u^{-\rho\textcolor{black}{-\frac{1}{2}}} v^{-\gamma'_i}\Big)
        \,\Qalt_{B}^H\, 
        \Gamma_{-}\Big(u^{-\gamma_i} v^{-\rho\textcolor{black}{+\frac{1}{2}}}\Big)
        \Gamma_{+}\Big( u^{-\rho\textcolor{black}{+\frac{1}{2}}} v^{-\gamma'_i\textcolor{black}{-1}}\Big)\!\!
    \right)
\end{align*}

The weight operator $Q^H$ commutes with the other operators with the following relations (see for example \cite[p. 125]{Young_Generating}):
\begin{align*}
    Q^H \,\Gamma_-(\mathbf{x}) 
    =
    \Gamma_-(Q\mathbf{x})\, Q^H
    \hspace{3em}
    \Gamma_+(\mathbf{x}) \, Q^H 
    =
    Q^H\, \Gamma_+(Q\mathbf{x})
\end{align*}
\begin{align*}
    Q^H \,\Gamma'_-(\mathbf{x}) 
    =
    \Gamma'_-(Q\mathbf{x})\, Q^H
    \hspace{3em}
    \Gamma'_+(\mathbf{x}) \, Q^H 
    =
    Q^H\, \Gamma'_+(Q\mathbf{x})
\end{align*}
Using these relations to commute the $\Qalt_{B}^H$ terms to the left of each of the ``$i$'' factors gives that $Z_{\bm{\gamma}}$ is the trace of the following expression:
\begin{align*}
        \prod_{i=0}^{r-1} 
        \,(Q_{B}Q_{A_i})^H\,
        \Gamma'_{-}\Big(u^{-\gamma_i} v^{-\rho\textcolor{black}{+\frac{1}{2}}}\Big) 
        \Gamma'_{+}\Big( u^{-\rho\textcolor{black}{-\frac{1}{2}}} v^{-\gamma'_i}\Big)
        \Gamma_{-}\Big(\Qalt_{B}u^{-\gamma_i} v^{-\rho\textcolor{black}{+\frac{1}{2}}}\Big)
        \Gamma_{+}\Big( \Qalt_{B}^{-1}u^{-\rho\textcolor{black}{+\frac{1}{2}}} v^{-\gamma'_i\textcolor{black}{-1}}\Big)
\end{align*}
Now, commuting the $(Q_{B}Q_{A_i})^H$ terms to the left and using the cyclic property of $\mathrm{tr}(\_)$ 
gives 
\begin{align}
    Z_{\bm{\gamma}}
    =
    \mathrm{tr}\left(
     p^H
    \prod_{i=0}^{r-1} 
    \Gamma'_{-}\big( \bm{\mathtt{v}_i} \big) 
    \,
    \Gamma'_{+}\big(  u^{-1}   \bm{\mathtt{u}_i}\big) 
    \,
    \Gamma_{-}\big( \Qalt_B \bm{\mathtt{v}_i}  \big)
    \,
    \Gamma_{+}\big(  \Qalt_B^{-1} v^{-1} \bm{\mathtt{u}_i}\big)
    \right)
    \label{eqn: Z_gamma trace formula}
\end{align}
where we have defined
$\mathtt{e}_i : = \prod_{j=0}^i (Q_{B}Q_{A_i})^{-1}$, 
$p:=(e_{r-1})^{-1}$
and the following vectors:
\begin{align}
    \bm{\mathtt{u}_i} 
    =
    \Big( \mathtt{e}_i^{-1}\, u^{n} v^{-(\gamma'_i)_n}   \Big)_{n>0} 
    \hspace{1em}
    \mbox{and}
    \hspace{1em}
    \bm{\mathtt{v}_i}
    =
    \Big( \mathtt{e}_i\, u^{-(\gamma_i)_n} v^{n} \Big)_{n>0}.
    \label{eqn: vectors DT}
\end{align}

\subsection{Computing elliptic genus from trace formula}
Consider an $r$-tuple of Young diagrams $\bm{\gamma} = (\gamma_0,\ldots,\gamma_{r-1})$ and define vectors
\[
    \ui{} 
    =
    \Big( e_i^{-1} t_1^{(\gamma'_i)_n} t_2^{n}   \Big)_{n>0} 
    \hspace{1em}
    \mbox{and}
    \hspace{1em}
    \vi{}
    =
    \Big( e_i\, t_1^{-n} t_2^{-(\gamma_i)_n} \Big)_{n>0}
\]
\emph{inspired by} those of \eqref{eqn: vectors DT}. Also consider the following trace expression inspired by Equation \eqref{eqn: Z_gamma trace formula}:
\begin{align}
    W_{\bm{\gamma}}
    &
    :=
    \mathrm{tr}\left(
     p^H
    \prod_{i=0}^{r-1} 
    \Gamma'_{-}\big( \vi{} \big)
    \, 
    \Gamma'_{+}\big(  t_2^{-1}  \ui{}\big) 
    \,
    \Gamma_{-}\big( -y\,  \vi{} \big)
    \,
    \Gamma_{+}\big( -y^{-1} t_1\,  \ui{}\big)
    \right)
    \label{eqn: W_gamma def}
\end{align}
Note that the right-hand side of Equation \ref{eqn: Z_gamma trace formula} can be obtained by substituting 
\[
    e_i  = \mathtt{e}_i,
    \hspace{1em}
    \hspace{1em}
    t_1 = \frac{1}{v},
    \hspace{1em}
    t_2 = u,
    \hspace{1em}
    \mbox{and}
    \hspace{1em}
    y= -\Qalt_B.
\]

We will compute a product expression for $W_{\bm{\gamma}}$ by employing the method of \cite[\S5]{BKY_Trace}, namely commuting the $\Gamma^*_\pm$ terms and obtaining an expressions with increasingly large powers of $p$. Then considering this expression $\mathrm{mod}(q^M)$ for $M\rightarrow \infty$. 
The key tools are the commutation relations, which recall from \cite[Lem. 3.3]{Young_Generating} as
\[
    \big[\Gamma_+(\bm{a}), \Gamma'_-(\bm{b})\big] 
    ~~=~~
    \big[\Gamma'_+(\bm{a}), \Gamma_-(\bm{b})\big]
    ~~=~~
    \prod_{n,m>0} (1+ a_n b_m )
    ~~~~
    \mbox{and}
\]
\[
    \big[\Gamma_+(\bm{a}), \Gamma_-(\bm{b})\big] 
    ~~=~~
    \big[\Gamma'_+(\bm{a}), \Gamma'_-(\bm{b})\big]
    ~~=~~
    \prod_{n,m>0} (1- a_n b_m )^{-1}, 
\]
where $\bm{a} = (a_n)_{n\in\NN}$ and $\bm{b} = (b_m)_{m\in\NN}$.

Now, commuting the 
$\Gamma'_{+}\big(  t_2^{-1}  \ui{}\big)$ terms 
with the 
$
\Gamma_{-}\big( -y\,  \vi{} \big)$ 
terms, followed by commuting the 
$
\Gamma'_{+}\big(  t_2^{-1}  \ui{}\big) 
\,
\Gamma_{+}\big( -y^{-1} t_1\, \ui{}\big)
$
terms to the right we obtain
\begin{align*}
    W_{\bm{\gamma}}
    =
    &~
    \prod_{i=0}^{r-1} \prod_{n,m} (1 - y\, t_2^{-1} \ui{,n}\, \vi{,m})
    \cdot
    \mathrm{tr}\left(
     p^H
    \prod_{i=0}^{r-1} 
    \Gamma'_{-}\big(   \vi{} \big) 
    \Gamma_{-}\big(  -y\,\vi{} \big)
    \Gamma'_{+}\big(  t_2^{-1} \ui{}\big) 
    \Gamma_{+}\big( -y^{-1} t_1\, \ui{}\big)
    \!
    \right)
    \\
    =&~
    \prod_{i=0}^{r-1} \prod_{n,m} (1 - y\, t_2^{-1} \ui{,n}\, \vi{,m})
    \cdot
    \prod_{0\leq i<j\leq r-1}
    \prod_{n,m} 
    \frac{(1 - y\, t_2^{-1} \uj{,n}\, \vi{,m})(1 - y^{-1} t_1\, \uj{,n}\, \vi{,m})}{(1 - t_2^{-1} \uj{,n}\, \vi{,m})(1 - t_1\, \uj{,n}\, \vi{,m})}
    \\
    &~
    \cdot
    \mathrm{tr}\left(
     p^H
    \prod_{j=0}^{r-1} 
    \Gamma'_{-}\big(   \vj{} \big) 
    \Gamma_{-}\big( -y\, \vj{} \big)
    \prod_{i=0}^{r-1} 
    \Gamma'_{+}\big(  t_2^{-1}  \ui{}\big) 
    \Gamma_{+}\big( -y^{-1} t_1\, \ui{}\big)
    \right)
\end{align*}

For a 2D partition, Y, with $(m,n)\in Y$ recall the arm and leg functions $a_Y(m,n) := Y_m - n$ and $l_Y(m,n) := (Y^T)_n - m$. Hence we have  
\begin{align}
    \uj{,n}\, \vi{,m} 
    =
    e_i\,e_j^{-1} 
    t_1^{(\gamma_j')_n - m}
    t_2^{-(\gamma_i)_m+n}
    =
    e_i\,e_j^{-1} 
    t_1^{l_{\gamma_j}(m,n)}
    t_2^{-a_{\gamma_i}(m,n)}
    =
    e_i\,e_j^{-1} 
    t_1^{a_{\gamma'_j}(n,m)}
    t_2^{-l_{\gamma'_i}(n,m)}.
    \label{eqn: uj vi product}
\end{align}
We use this to consider the factors of the above expression of $W_{\bm{\gamma}}$.

\begin{lem}
    \label{lem: W_gamma factor 1}
    We have the equality
    \begin{align*}
        &
        \prod_{i=0}^{r-1} \prod_{n,m} (1 - y\, t_2^{-1} \ui{,n}\, \vi{,m})
        \cdot
        \prod_{0\leq i<j\leq r-1}
        \prod_{n,m} 
        \frac{(1 - y\, t_2^{-1} \uj{,n}\, \vi{,m})(1 - y^{-1} t_1\, \uj{,n}\, \vi{,m})}{(1 - t_2^{-1} \uj{,n}\, \vi{,m})(1 - t_1\, \uj{,n}\, \vi{,m})}
        \\
        &
        =
        y^{(1-r) |\bm{\gamma}|}
        \prod_{0\leq i<j\leq r-1}
        \frac{M(e_i\,e_j^{-1}t_1\,; t_1^{-1},t_2)}{M( y^{-1} e_i\,e_j^{-1}t_1\,; t_1^{-1},t_2)}
        \frac{M( y\,e_j\,e_i^{-1}t_2^{-1}; t_1^{-1},t_2)}{M(e_j\,e_i^{-1}t_2^{-1}; t_1^{-1},t_2)}
        \\
        &
        ~~~\cdot
        \prod_{i=0}^{r-1} \prod_{n,m} (1 -  t_2^{-1} \ui{,n}\, \vi{,m})
        \cdot
        \prod_{i,j=0}^{r-1}
        \prod_{n,m} 
        \frac{(1 - y\, t_2^{-1} \uj{,n}\, \vi{,m})}{(1 - t_2^{-1} \uj{,n}\, \vi{,m})}.
    \end{align*}
\end{lem}
\begin{proof}
    We refer to Equation \eqref{eqn: uj vi product} so that
    using Corollary \ref{cor: infinite to finite product} 
    with $k = y^{-1}e_i\,e_j^{-1}$, then inverting the monomials and reordering the products; before using 
    Corollary \ref{cor: infinite to finite product} again 
    with $k = y\,e_j\,e_i^{-1}t_1\,t_2$ gives
    \begin{align*}
        &
        \frac{M( y^{-1} e_i\,e_j^{-1}t_1; t_1^{-1},t_2)}{M(e_i\,e_j^{-1}t_1; t_1^{-1},t_2)}
        \cdot
        \prod_{n,m} 
        \frac{(1 - y^{-1} t_1\, \uj{,n}\, \vi{,m})}{(1 - t_1\,\uj{,n}\, \vi{,m})}
        \\
        &
        =
        \prod_{x\in\gamma'_i}
        \frac{\Big(1 - y^{-1}e_i\,e_j^{-1} 
        t_1^{a_{\gamma'_j}(x)+1}
        t_2^{-l_{\gamma'_i}(x)}\Big)}{\Big(1 - e_i\,e_j^{-1} 
        t_1^{a_{\gamma'_j}(x)+1}
        t_2^{-l_{\gamma'_i}(x)}\Big)}
        \prod_{x\in\gamma'_j} 
        \frac{\Big(1 - y^{-1}e_i\,e_j^{-1} 
        t_1^{-a_{\gamma'_j}(x)}
        t_2^{l_{\gamma'_i}(x)+1}\Big)}{\Big(1 - e_i\,e_j^{-1} 
        t_1^{-a_{\gamma'_j}(x)}
        t_2^{l_{\gamma'_i}(x)+1}\Big)}
        \\
        &
        =
        y^{-|\gamma'_j|-|\gamma'_i|}
        \prod_{x\in\gamma'_j} 
        \frac{\Big(1 - y\,e_j\,e_i^{-1} t_1^{a_{\gamma'_j}(x)}
        t_2^{-l_{\gamma'_i}(x)-1}\Big)}{\Big(1 - e_j\,e_i^{-1} t_1^{a_{\gamma'_j}(x)}
        t_2^{-l_{\gamma'_i}(x)-1}\Big)}
        \prod_{x\in\gamma'_i}
        \frac{\Big(1 - y\,e_j\,e_i^{-1} t_1^{-a_{\gamma'_j}(x)-1}
        t_2^{l_{\gamma'_i}(x)}\Big)}{\Big(1 - e_j\,e_i^{-1} t_1^{-a_{\gamma'_j}(x)-1}
        t_2^{l_{\gamma'_i}(x)}\Big)}
        \\
        &
        =
        y^{-|\gamma'_j|-|\gamma'_i|}
        \cdot 
        \frac{M( y\,e_j\,e_i^{-1}t_2^{-1}; t_1^{-1},t_2)}{M(e_j\,e_i^{-1}t_2^{-1}; t_1^{-1},t_2)}
        \cdot
        \prod_{n,m} 
        \frac{(1 - y\, t_2^{-1} \ui{,n}\, \vj{,m})}{(1 - t_2^{-1} \ui{,n}\, \vj{,m})}.
    \end{align*}
    Finally, in the set $\{\,(i,j)~|~ 0\leq i<j\leq r-1\,\}$ each element of $\{0,\ldots, r-1\}$ appears $r-1$ times. Hence, we have
    \[
        \sum_{{0\leq i<j\leq r-1}}\big( |\gamma'_j|+|\gamma'_i|\big) =  (r-1) \sum_{i=0}^{r-1} |\gamma_j|
    \]
    and the desired result now follows. 
\end{proof}

\begin{lem}
    \label{lem: W_gamma factor 2}
    We have the equality
    \begin{align*}
        \mathsf{Tr}
        :=~
        &
        \mathrm{tr}\left(
        p^H
        \prod_{j=0}^{r-1} 
        \Gamma'_{-}\big(   \vj{} \big) 
        \Gamma_{-}\big( -y\, \vj{} \big)
        \prod_{i=0}^{r-1} 
        \Gamma'_{+}\big(  t_2^{-1}  \ui{}\big) 
        \Gamma_{+}\big( -y^{-1} t_1\, \ui{}\big)
        \right)
        \\
        =~
        &
        \prod_{N>0} \frac{1}{1-p^N}
        \prod_{i,j=0}^{r-1} \prod_{n,m>0} 
        \frac{(1 - p^N y\, t_2^{-1} \uj{,n}\, \vi{,m})(1 - p^N y^{-1} t_1\, \uj{,n}\, \vi{,m})}{(1 -p^N t_2^{-1} \uj{,n}\, \vi{,m})(1 - p^N t_1\, \uj{,n}\, \vi{,m})}
    \end{align*}
\end{lem}
\begin{proof}
    Using the cyclic property of the trace, we move the
    $
        \prod_{i=0}^{r-1} 
        \Gamma'_{+}\big(  t_2^{-1}  \ui{}\big) 
        \Gamma_{+}\big( -y^{-1} t_1\, \ui{}\big)
    $
    term to the left and then commute it back to the right using the commutation relations. 
    Repeating this $M$ times gives
    \begin{align*}
        \mathsf{Tr}
        =
        &
        \prod_{i,j=0}^{r-1} \prod_{N=1}^{M} \prod_{n,m>0} 
        \frac{(1 - p^N y\, t_2^{-1} \uj{,n}\, \vi{,m})(1 - p^N y^{-1} t_1\, \uj{,n}\, \vi{,m})}{(1 -p^N t_2^{-1} \uj{,n}\, \vi{,m})(1 - p^N t_1\, \uj{,n}\, \vi{,m})}
        \\
        &
        \cdot
        \mathrm{tr}\left(
        p^H
        \prod_{j=0}^{r-1} 
        \Gamma'_{-}\big(   \vj{} \big) 
        \Gamma_{-}\big( -y\, \vj{} \big)
        \prod_{i=0}^{r-1} 
        \Gamma'_{+}\big(  p^M t_2^{-1}  \ui{}\big) 
        \Gamma_{+}\big( -p^M y^{-1} t_1\, \ui{}\big)
        \right).
    \end{align*}
    Considering this $\mathrm{mod}(q^M)$ implies that
    \begin{align*}
        \mathsf{Tr}
        =
        &
        \prod_{i,j=0}^{r-1} \prod_{N>0} \prod_{n,m>0} 
        \frac{(1 - p^N y\, t_2^{-1} \uj{,n}\, \vi{,m})(1 - p^N y^{-1} t_1\, \uj{,n}\, \vi{,m})}{(1 -p^N t_2^{-1} \uj{,n}\, \vi{,m})(1 - p^N t_1\, \uj{,n}\, \vi{,m})}
        \\
        &
        \cdot
        \mathrm{tr}\left(
        p^H
        \prod_{j=0}^{r-1} 
        \Gamma'_{-}\big(   \vj{} \big) 
        \Gamma_{-}\big( -y\, \vj{} \big)
        \right).
    \end{align*}
    Now, consider the remaining trace factor and commute $p^H$ term to the right and then use the cyclic property to move it back to the left. Repeating this $M$ times gives
    \[
        \mathrm{tr}\left(
        p^H
        \prod_{j=0}^{r-1} 
        \Gamma'_{-}\big(   \vj{} \big) 
        \Gamma_{-}\big( -y\, \vj{} \big)
        \right)
        =
        \mathrm{tr}\left(
        p^H
        \prod_{j=0}^{r-1} 
        \Gamma'_{-}\big(  p^M \vj{} \big) 
        \Gamma_{-}\big( -p^M y\, \vj{} \big)
        \right)
    \]
    Hence, this term is equal to 
    $
        \mathrm{tr}\left(
        p^H
        \right)
        =
        \sum_{\mu} \left<\mu \,|\, p^{H} \,|\, \mu\right>
        = 
        \prod_{N>0} \left(1-p^N\right)^{-1}, 
    $
    the partition function for the number of 2D partitions. 
\end{proof}

\begin{lem}
    \label{lem: Theta products}
    We have the equality
    \begin{align*}
        &
        \prod_{N>0}
        \prod_{i,j=0}^{r-1} \prod_{n,m>0} 
        \frac{(1 - p^{N-1} y\, t_2^{-1} \uj{,n}\, \vi{,m})(1 - p^N y^{-1} t_1\, \uj{,n}\, \vi{,m})}{(1 -p^{N-1} t_2^{-1} \uj{,n}\, \vi{,m})(1 - p^N t_1\, \uj{,n}\, \vi{,m})}
        \\
        &
        =
        \prod_{N>0}
        \prod_{i,j=0}^{r-1} 
        \frac{M(p^{N-1}e_i\,e_j^{-1}t_2^{-1}; t_1^{-1},t_2)}{M(p^{N-1} y\,\,e_i\,e_j^{-1} t_2^{-1}; t_1^{-1},t_2)}
        \frac{M(p^{N}e_i\,e_j^{-1}t_1\,; t_1^{-1},t_2)}{M(p^{N} y^{-1}e_i\,e_j^{-1} t_1\,; t_1^{-1},t_2)}
        \\
        &
        ~~~
        \cdot
        \prod_{i,j=0}^{r-1}
        ~
        \prod_{x \in \gamma_i}  
        \Theta_{p,y} \left( \dfrac{e_i}{e_j} \dfrac{t_1^{l_{\gamma_j}(x)}}{t_2^{a_{\gamma_i}(x)+1}} \right)
        ~
        \prod_{z \in \gamma_j}  
        \Theta_{p,y} \left( \dfrac{e_i}{e_j} \dfrac{t_2^{a_{\gamma_j}(z)}}{t_1^{l_{\gamma_i}(z)+1}}  \right)
    \end{align*}
    where we use the notation 
    \[
        \Theta_{p,y}(t) 
        :=
        \prod_{n>0} 
            \dfrac{(1 - y\,p^{n-1} t)}{(1 - p^{n-1} t)}
            \dfrac{(1 - y^{-1} p^{n} t^{-1})}{(1 - p^{n} t^{-1} )}.
    \]
\end{lem}
\begin{proof}
    We recall Equation \eqref{eqn: uj vi product} use and Corollary \ref{cor: infinite to finite product} to give
    \begin{align*}
        &
        \prod_{i,j=0}^{r-1}
        \prod_{N>0}
        \frac{M(p^{N-1} y\,\,e_i\,e_j^{-1} t_2^{-1}; t_1^{-1},t_2)}{M(p^{N-1}e_i\,e_j^{-1}t_2^{-1}; t_1^{-1},t_2)}
        \frac{M(p^{N} y^{-1}e_i\,e_j^{-1} t_1\,; t_1^{-1},t_2)}{M(p^{N}e_i\,e_j^{-1}t_1\,; t_1^{-1},t_2)}
        \\
        &
        ~~
        \cdot
        \prod_{N>0} \prod_{n,m} 
        \frac{(1 - p^{N-1} y\, t_2^{-1} \uj{,n}\, \vi{,m})}{(1 -p^{N-1} t_2^{-1} \uj{,n}\, \vi{,m})}\frac{(1 - p^{N} y^{-1} t_1\, \uj{,n}\, \vi{,m})}{(1 - p^{N} t_1\, \uj{,n}\, \vi{,m})}
        \\
        &
        =
        \prod_{i,j=0}^{r-1}
        \prod_{N>0} 
        \prod_{x\in \gamma_i} 
        \frac{\Big(1 - p^{N-1} y\, e_i\,e_j^{-1} t_1^{l_{\gamma_j}(x)}
        t_2^{-a_{\gamma_i}(x)-1}\Big)}{\Big(1 - p^{N-1}  e_i\,e_j^{-1} t_1^{l_{\gamma_j}(x)}
        t_2^{-a_{\gamma_i}(x)-1}\Big)}
        \prod_{z\in \gamma_j} 
        \frac{\Big(1 - p^{N-1} y\, e_i\,e_j^{-1} t_1^{-l_{\gamma_i}(x)-1}
        t_2^{a_{\gamma_j}(x)}\Big)}{\Big(1 - p^{N-1} e_i\,e_j^{-1} t_1^{-l_{\gamma_i}(x)-1}
        t_2^{a_{\gamma_j}(x)}\Big)}
        \\
        &
        \hspace{5.15em}
        \prod_{x\in \gamma_i} 
        \frac{\Big(1 - p^{N} y^{-1} e_i\,e_j^{-1} t_1^{l_{\gamma_j}(x)+1}
        t_2^{-a_{\gamma_i}(x)}\Big)}{\Big(1 - p^{N}  e_i\,e_j^{-1} t_1^{l_{\gamma_j}(x)+1}
        t_2^{-a_{\gamma_i}(x)}\Big)}
        \prod_{z\in \gamma_j} 
        \frac{\Big(1 - p^{N} y^{-1} e_i\,e_j^{-1} t_1^{-l_{\gamma_i}(x)}
        t_2^{a_{\gamma_j}(x)+1}\Big)}{\Big(1 - p^{N} e_i\,e_j^{-1} t_1^{-l_{\gamma_i}(x)}
        t_2^{a_{\gamma_j}(x)+1}\Big)}
        \\
        &
        =
        \prod_{i,j=0}^{r-1}
        \prod_{N>0} 
        \prod_{x\in \gamma_i} 
        \frac{\Big(1 - p^{N-1} y\, e_i\,e_j^{-1} t_1^{l_{\gamma_j}(x)}
        t_2^{-a_{\gamma_i}(x)-1}\Big)}{\Big(1 - p^{N-1}  e_i\,e_j^{-1} t_1^{l_{\gamma_j}(x)}
        t_2^{-a_{\gamma_i}(x)-1}\Big)}
        \prod_{z\in \gamma_j} 
        \frac{\Big(1 - p^{N-1} y\, e_i\,e_j^{-1} t_1^{-l_{\gamma_i}(x)-1}
        t_2^{a_{\gamma_j}(x)}\Big)}{\Big(1 - p^{N-1} e_i\,e_j^{-1} t_1^{-l_{\gamma_i}(x)-1}
        t_2^{a_{\gamma_j}(x)}\Big)}
        \\
        &
        \hspace{5.15em}
        \prod_{x\in \gamma_{\textcolor{black}{j}}} 
        \frac{\Big(1 - p^{N} y^{-1} e_{\textcolor{black}{j}}\,e_{\textcolor{black}{i}}^{-1} t_1^{l_{\gamma_{\textcolor{black}{i}}}(x)+1}
        t_2^{-a_{\gamma_{\textcolor{black}{j}}}(x)}\Big)}{\Big(1 - p^{N}  e_{\textcolor{black}{j}}\,e_{\textcolor{black}{i}}^{-1} t_1^{l_{\gamma_{\textcolor{black}{i}}}(x)+1}
        t_2^{-a_{\gamma_{\textcolor{black}{j}}}(x)}\Big)}
        \prod_{z\in \gamma_{\textcolor{black}{i}}} 
        \frac{\Big(1 - p^{N} y^{-1} e_{\textcolor{black}{j}}\,e_{\textcolor{black}{i}}^{-1} t_1^{-l_{\gamma_{\textcolor{black}{j}}}(x)}
        t_2^{a_{\gamma_{\textcolor{black}{i}}}(x)+1}\Big)}{\Big(1 - p^{N} e_{\textcolor{black}{j}}\,e_{\textcolor{black}{i}}^{-1} t_1^{-l_{\gamma_{\textcolor{black}{j}}}(x)}
        t_2^{a_{\gamma_{\textcolor{black}{i}}}(x)+1}\Big)}
        \\
        &
        =
        \prod_{i,j=0}^{r-1}
        \prod_{x \in \gamma_i}  
        \Theta_{p,y} \left( \dfrac{e_i}{e_j} \dfrac{t_1^{l_{\gamma_j}(x)}}{t_2^{a_{\gamma_i}(x)+1}} \right)
        ~
        \prod_{z \in \gamma_j}  
        \Theta_{p,y} \left( \dfrac{e_i}{e_j} \dfrac{t_2^{a_{\gamma_j}(z)}}{t_1^{l_{\gamma_i}(z)+1}}  \right).
    \end{align*}

\end{proof}

Hence, combining Lemmas \ref{lem: W_gamma factor 1}, \ref{lem: W_gamma factor 2} and \ref{lem: Theta products} with Definition-Theorem \ref{defthm: M(r,N)}
we arrive at the following formula for the elliptic genus of the moduli space of framed instanton sheaves. 

\begin{cor}
    \label{cor: elliptic genus as trace final}
    For an $r$-tuple $\bm{\gamma}=(\gamma_0,\ldots,\gamma_{r-1})$ of 2D partitions, define the vectors 
    \[
        \ui{} 
        =
        \Big( e_i^{-1} t_1^{(\gamma'_i)_n} t_2^{n}   \Big)_{n>0} 
        \hspace{1em}
        \mbox{and}
        \hspace{1em}
        \vi{}
        =
        \Big( e_i\, t_1^{-n} t_2^{-(\gamma_i)_n} \Big)_{n>0}.
    \]
    Then we have 
    \begin{align*}
        &
        \mathrm{Ell}_{p,y}(\M(r,n); e_0,\ldots,e_{r-1},t_1,t_2 )\\
        &
        = 
        \frac{1}{y^n}
        \prod_{0\leq i<j\leq r-1}
        \frac{M( y^{-1} e_i\,e_j^{-1}t_1\,; t_1^{-1},t_2)}{M(e_i\,e_j^{-1}t_1\,; t_1^{-1},t_2)}
        \frac{M(e_j\,e_i^{-1}t_2^{-1}; t_1^{-1},t_2)}{M( y\,e_j\,e_i^{-1}t_2^{-1}; t_1^{-1},t_2)}
        \\
        &
        ~~~
        \cdot
        \prod_{N>0}
        (1-p^N)
        \prod_{i,j=0}^{r-1} 
        \frac{M(p^{N-1} y\,\,e_i\,e_j^{-1} t_2^{-1}; t_1^{-1},t_2)}{M(p^{N-1}e_i\,e_j^{-1}t_2^{-1}; t_1^{-1},t_2)}
        \frac{M(p^{N} y^{-1}e_i\,e_j^{-1} t_1\,; t_1^{-1},t_2)}{M(p^{N}e_i\,e_j^{-1}t_1\,; t_1^{-1},t_2)}
        \\
        &
        ~~~
        \cdot
        \multisum{\mbox{\scriptsize$\bm{\gamma}$ with} \\ |\bm{\gamma}|= n}
        \frac{
        \mathrm{tr}\left(
        p^H
        \prod\limits_{i=0}^{r-1} 
        \Gamma'_{-}\big( \vi{} \big)
        \, 
        \Gamma'_{+}\big(  t_2^{-1}  \ui{}\big) 
        \,
        \Gamma_{-}\big( -y\,  \vi{} \big)
        \,
        \Gamma_{+}\big( -y^{-1} t_1\,  \ui{}\big)
        \right)}
        {\prod\limits_{i=0}^{r-1} \prod\limits_{n,m} (1 -  t_1\, \ui{,n}\, \vi{,m})}.
    \end{align*}
\end{cor}

\subsection{Proof of the main theorem and corollaries}

We recall from Equations \eqref{eqn: Z_gamma trace formula} and \eqref{eqn: W_gamma def} that $Z_{\bm{\gamma}}$ can be obtained from $W_{\bm{\gamma}}$ by substituting 
\[
    e_i  = \mathtt{e}_i,
    \hspace{1em}
    p=  = \mathtt{e}_{r-1},
    \hspace{1em}
    t_1 = \frac{1}{v},
    \hspace{1em}
    t_2 = u,
    \hspace{1em}
    \mbox{and}
    \hspace{1em}
    y= -\Qalt_B.
\]
Moreover, from Equation \eqref{eqn: Z^DT Y-Z split} we have
\begin{align}
    Z_c^{\mathbb{DT}}
    =
    &
    \multisum{\mbox{\scriptsize$\bm{\gamma}$ with} \\ |\bm{\gamma}|= c}
    Y_{\gamma_0}\cdots Y_{\gamma_{r-1}}\cdot Z_{\bm{\gamma}}.
\end{align}
Hence we obtain the expressions 
\begin{align*}
    Z_c^{\mathbb{DT}}
    =
    &
    ~
    M(v^{-1};u,v)^{r}
    \cdot
    \prod_{0\leq i<j\leq r-1}
    \frac{M\big(\mathtt{e}_i\,\mathtt{e}_j^{-1}v^{-1}; u,v\big)}{M\big( -\Qalt_B^{-1} \mathtt{e}_i\,\mathtt{e}_j^{-1}v^{-1}; u,v\big)}
    \frac{M\big( -\Qalt_B\,\mathtt{e}_j\,\mathtt{e}_i^{-1}u^{-1}; u,v\big)}{M\big(\mathtt{e}_j\,\mathtt{e}_i^{-1}u^{-1}; u,v\big)}
    \\
    &
    ~
    \cdot
    \prod_{N>0}
    \frac{1}{1-e_{r-1}^N}
    \prod_{i,j=0}^{r-1} 
    \frac{M\big(\mathtt{e}_{r-1}^{N-1}\mathtt{e}_i\,\mathtt{e}_j^{-1}u^{-1}; u,v\big)}{M\big(-\Qalt_B\,\mathtt{e}_{r-1}^{N-1} \mathtt{e}_i\,\mathtt{e}_j^{-1} u^{-1}; u,v\big)}
    \frac{M\big(\mathtt{e}_{r-1}^{N}\mathtt{e}_i\,\mathtt{e}_j^{-1}v^{-1}; u,v\big)}{M\big(-\Qalt_B^{-1}\mathtt{e}_{r-1}^{N} \mathtt{e}_i\,\mathtt{e}_j^{-1} v^{-1}; u,v\big)}
    \\
    &
    ~
    \cdot
    \left(\frac{v}{u}\right)^{\frac{1}{2} c}\Qalt_B^c
    \multisum{\mbox{\scriptsize$\bm{\gamma}$ with} \\ |\bm{\gamma}|= c}
    (-\Qalt_B)^{-r|\bm{\gamma}|}
    \prod_{i,j=0}^{r-1}
    ~
    \prod_{x \in \gamma_i}  
    \Theta_{\mathtt{e}_{r-1},-\Qalt_B} \left( \dfrac{\mathtt{e}_i}{\mathtt{e}_j} u^{-a_{\gamma_i}(x)-1} v^{-l_{\gamma_j}(x)} \right)
    \\
    &
    \hspace{16.06em}
    \prod_{z \in \gamma_j}  
    \Theta_{\mathtt{e}_{r-1},-\Qalt_B} \left( \dfrac{\mathtt{e}_i}{\mathtt{e}_j} u^{a_{\gamma_j}(z)} v^{l_{\gamma_i}(z)+1}  \right)
    \\
    =
    &
    ~
    \calZ_0^{\mathbb{DT}}
    \cdot
    \left(\frac{v}{u}\right)^{\frac{1}{2} c}\Qalt_B^c
    \cdot
    \Ell_{\mathtt{e}_{r-1},-\Qalt_B}\big( \M(r,c); \mathtt{e}_0,\ldots \mathtt{e}_{r-1}, v^{-1}, u \big),
\end{align*}
where we have defined $\calZ_0^{\mathbb{DT}}:= Z_0^{\mathbb{DT}}$,
proving Theorem \ref{thm: main} and Corollary \ref{cor: Partition function as elliptic genus}.\\

\subsection{Product formula for $r=1$ case}
We now consider the $r=1$ and use the techniques of \cite[\S2 and \S5.3]{Bryan_DonaldsonThomas}. In this case we have
\begin{align*}
    \Ell_{\mathtt{e}_{0},-\Qalt_B}\big( \M(1,c); \mathtt{e}_0, v^{-1}, u \big)
    =
    \mathrm{Ell}_{\mathtt{e}_0,-\Qalt_B} \big(\Hilb^n(\CC^2); v^{-1},u \big),
\end{align*}
where $\Hilb^n(\CC^2)$ is the Hilbert scheme of $n$ points on $\CC^2$. 
In this special case we can use Waelder's analogue of the DMVV formula where $(\CC^*)^2$ acts with weight $(-1,1)$:

\begin{thm}
    [\mbox{\cite[Thm. 12]{Waelder_Equivariant}}]
    \label{thm: Waelder DMVV}
    Considering the Laurent expansion of
    $
        \Ell_{p,y}\big(\CC^2; t_1^{-1},t_2\big)
    $
    in $p$, $y$, $t_1^{-1}$, $t_2$ 
    \[
        \sum_{m\geq 0} \sum_{l,k_1,k_2\in \ZZ}
        c(m,l,k_1,k_2)\, p^m y^l t_1^{-k_1} t_2^{k_2},
    \]
    we have the following formula
    \[
        \sum_{n \geq 0} 
        \Ell_{p,y}\big(\Hilb^n(\CC^2); t_1^{-1},t_2\big)
        \,Q^{n}
        = 
        \prod_{m,n\geq 0}
        \prod_{l,k_1,k_2\in \ZZ}
        \Big(1 - p^n Q^m y^l t_1^{-k_1} t_2^{k_2}\Big)^{-c(nm,l,k_1,k_2)}.
    \]
\end{thm}

Now define 
$\Qalt_{C} := Q_{C} \LL^{-\frac{1}{2}} = Q_{C} u^{-1/2} v^{1/2}$
and consider
\begin{align*}
    \frac{Z^{\mathbb{DT}}}{\calZ_0^{\mathbb{DT}}}
    =
    ~
    &
    \sum_{c\geq 0}
    \mathrm{Ell}_{\,\Qalt_{A_0}\Qalt_B,-\Qalt_B} \big(\Hilb^n(\CC^2); v^{-1},u \big)
    \,
    \Qalt_B^c
    \,
    \Qalt_C^c
    \\
    =
    ~
    &
    \prod_{n\geq 0}
    \prod_{m>0}
    \prod_{l,k_1,k_2\in \ZZ}
    \Big(1 - (Q_{A_0}Q_B)^n ( Q_B Q_{C} )^m \big(-Q_B (\tfrac{u}{v})^{\frac{1}{2}}\big)^l v^{-k_1} u^{k_2}\Big)^{-c(nm,l,k_1,k_2)}
    \\
    =
    ~
    &
    \prod_{n\geq 0}
    \prod_{m>0}
    \prod_{l,k_1,k_2\in \ZZ}
    \Big(1 - (-Q_{A_0})^n (-Q_B)^{m+n+l} (-Q_{C})^m v^{-k_1-\frac{1}{2}l} u^{k_2+\frac{1}{2}l}\Big)^{-c(nm,l,k_1,k_2)}
\end{align*}

Now consider the following generalisation of \cite[Prop. 2.2]{Bryan_DonaldsonThomas}. 

\begin{lem}
    \label{lem: Ell(CC) expansion}
    Consider the expansion of 
    $
        \Ell_{p,y}(\CC^2; t_1,t_2)
    $
    from Theorem \ref{thm: Waelder DMVV} in and define integers $b(m,l,k_1,k_2)$ via the expansions as Laurent series in $p,y,t_1^{_{-1/2}},t_2^{_{1/2}}$ 
    \begin{align*}
        \sum_{m\geq 0} \sum_{l,k_1,k_2\in \ZZ}
        c(m,l,k_1,k_2)\, p^m y^l t_1^{-k_1} t_2^{k_2}
        &
        =
        \multisum{m\geq 0, l\in \ZZ \\ k_1,k_2\in \frac{1}{2}\ZZ}
        b(m,l,k_1,k_2)\, p^m \left(\frac{y}{(t_1 t_2)^{\frac{1}{2}}}\right)^l t_1^{-k_1} t_2^{k_2}.
    \end{align*}
    Then $b(m,l,k_1,k_2)$ depends only on $(4m-l^2, k_1,k_2)$. Writing
    \[
        b(m,l,k_1,k_2) = b(4m-l^2, k_1,k_2)
    \]
    we have $b(a,k_1,k_2) = 0 $ if $a< -1$ and the Laurent series in $Q,t_1^{_{-1/2}},t_2^{_{1/2}}$ 
    \[
        \sum_{a\geq -1} \sum_{k_1,k_2\in \frac{1}{2}\ZZ}
        b(a,k_1,k_2) 
        Q^a t_1^{-k_1} t_2^{k_2}
        =
        \frac{
            \sum_{k\in \ZZ} Q^{k^2} \big(-(t_1t_2)^{\frac{1}{2}}\big)^k
        }{
            \left(
                \sum_{k\in \ZZ+\frac{1}{2}} Q^{2k^2} (-t_1)^{-k}
            \right)
            \left(
                \sum_{k\in \ZZ+\frac{1}{2}} Q^{2k^2} (-t_2)^k
            \right)
        }
    \]
\end{lem}
\begin{proof}
    The proof is essentially the same as that of \cite[Prop. 2.2]{Bryan_DonaldsonThomas}, but we include it for completeness. 
    By definition we have 
    \begin{align*}
        \Ell_{p,y}(\CC^2; t_1^{-1},t_2)
        =
        y^{-1}
        \Theta_{p,y}(t_1^{-1})
        \Theta_{p,y}(t_2)
        =
        \frac{\theta(p, y t_1)}{\theta(p, t_1)}
        \frac{\theta(p, y t_2^{-1})}{\theta(p, t_2^{-1})},
    \end{align*}
    and considering the denominator we have an expansion
    \[
        \frac{1}{\theta(p, t_1)\theta(p, t_2^{-1})}
        =
        q^{-\frac{1}{4}} 
        \frac{t_1^{-\frac{1}{2}}}{1-t_1^{-1}}
        \frac{t_2^{\frac{1}{2}}}{1-t_2}
        \sum_{i\geq 0} \delta_i(t_1,t_2) p^i
    \]
    for some $\delta_i(t_1,t_2)\in \ZZ[t_1^{\pm 1},t_2^{\pm 1}]$. 
    The numerator can be expanded using the Jacobi triple product formula to give
    \begin{align*}
        \theta(p, y t_1)
        \theta(p, y t_2^{-1})
        &
        =
        \sum_{n,m\in\ZZ}
        p^{\frac{1}{2}\big(n+\frac{1}{2}\big)^2+\frac{1}{2}\big(m+\frac{1}{2}\big)^2}
        (-yt_1)^{n+\frac{1}{2}}
        (-yt_2^{-1})^{m+\frac{1}{2}}
        \\
        &
        =
        \multisum{l,b\in\ZZ \\ l\equiv b+1~ \mathrm{mod}\,2}
        \!\!\!
        p^{\frac{1}{4}(l^2+b^2-1)}
        y^l
        (-t_1)^{\frac{1}{2}(l+b)}
        (-t_2)^{-\frac{1}{2}(l-b)}
        \\
        &
        =
        \multisum{l,b\in\ZZ \\ l\equiv b+1~ \mathrm{mod}\,2}
        \!\!\!
        p^{\frac{1}{4}(l^2+b^2-1)}
        \left(\frac{y\, t_1^{\frac{1}{2}} }{t_2^{\frac{1}{2}}}\right)^{\!\!l}
        \big( -(t_1 t_2)^{\frac{1}{2}} \big)^{b}
    \end{align*}
    where we have used the change of variables $l= m+n+1$ and $b=n-m$. 
    We then obtain the expression for $\Ell_{p,y}(\CC^2; t_1^{-1},t_2)$
    \begin{align*}
        q^{-\frac{1}{4}} 
        \frac{t_1^{-\frac{1}{2}}}{1-t_1^{-1}}
        \frac{t_2^{\frac{1}{2}}}{1-t_2}
        \!\!
        \multisum{i \geq0,~ l,b\in\ZZ \\ l\equiv b+1~ \mathrm{mod}\,2}
        \!\!\!
        \delta_i(t_1,t_2)
        \,\,
        p^{i+\frac{1}{4}(l^2+b^2-1)}
        \left(\frac{y\, t_1^{\frac{1}{2}} }{t_2^{\frac{1}{2}}}\right)^{\!\!l}
        \big( -(t_1 t_2)^{\frac{1}{2}} \big)^{b},
    \end{align*}
    and observe that the terms $q^{n} \big(y\, t_1^{_{1/2}}t_2^{_{-1/2}}\big)^l$ occur when $n = i+\frac{1}{4}(l^2+b^2-1)$.
    Thus they occur when $4n-l^2 = 4i+b^2 -1 \geq -1$, and the coefficient of $q^{n} \big(y\, t_1^{_{1/2}}t_2^{_{-1/2}}\big)^l$ depends only on this value, and is zero when this value is less than $-1$. 
    
    In summary 
    \begin{align*}
        \mathrm{Coeff}_{p^m \left(y\, t_1^{_{1/2}}t_2^{_{-1/2}}\right)^l}
        \left[ \Ell_{p,y}(\CC^2;t_1^{-1},t_2) \right]
        =
        \sum_{k_1,k_2\in \frac{1}{2}\ZZ}
        b(4m-l^2,k_1,k_2)\, t_1^{-k_1} t_2^{k_2}
    \end{align*}
    where 
    \begin{align*}
        \sum_{a\geq -1}
        \sum_{k_1,k_2\in \frac{1}{2}\ZZ}
        b(a,k_1,k_2)\, t_1^{-k_1} t_2^{k_2}
        &
        =
        -
        \frac{t_1^{-\frac{1}{2}}}{1-t_1^{-1}}
        \frac{t_2^{\frac{1}{2}}}{1-t_2}
        \sum_{i\geq 0} 
        \sum_{b\in\ZZ}
        Q^{4i+b^2-1}
        \big( -(t_1 t_2)^{\frac{1}{2}} \big)^{b}
        \delta_i(t_1,t_2)
        \\
        &
        =
        -Q^{-1}
        \frac{t_1^{-\frac{1}{2}}}{1-t_1^{-1}}
        \frac{t_2^{\frac{1}{2}}}{1-t_2}
        \sum_{i\geq 0} 
        \delta_i(t_1,t_2)
        Q^{4i}
        \sum_{b\in\ZZ}
        Q^{b^2}
        \big( -(t_1 t_2)^{\frac{1}{2}} \big)^{b}
        \\
        &
        =
        \theta_1(Q^4,t_1)^{-1}
        \theta_1(Q^4,t_2^{-1})^{-1}
        \sum_{b\in\ZZ}
        Q^{b^2} 
        \big( -(t_1 t_2)^{\frac{1}{2}} \big)^{b}
    \end{align*}
\end{proof}

\begin{cor}
    \label{cor: b terms for d_C =0}
    The coefficients $b(a,k_1,k_2)$ for $a = -1,0$ are given by
    \begin{align*}
        b(-1,k_1,k_2) 
        &
        =
        \left\{
            \begin{array}{ll}
                -1 &  \mbox{if $k_1\in \ZZ_{\leq 0} -\frac{1}{2}$ and $k_2 \in \ZZ_{\geq 0} +\frac{1}{2}$,}
                \\
                0 & \mbox otherwise.
            \end{array}
        \right.
        \\
        b(0,k_1,k_2) 
        &
        =
        \left\{
            \begin{array}{ll}
                1 & \mbox{if $k_1 = 0$ and $k_2 \in \ZZ_{\geq 0}$, or $k_1 \in \ZZ_{\leq 0}$ and $k_2 = 0$;}
                \\
                2 &  \mbox{if $k_1 \in \ZZ_{> 0}$ and $k_2 \in \ZZ_{> 0}$;}
                \\
                0 & \mbox{otherwise}.
            \end{array}
        \right.
    \end{align*}
\end{cor}
\begin{proof}
    Expanding in $t_1^{-1}$ and $t_2$, the result follows immediately from 
    \[
        \sum_{k_1,k_2 \in \frac{1}{2}\ZZ} b(-1,k_1,k_2)
        =  \frac{t_1^{-\frac{1}{2}}\,t_2^{\frac{1}{2}}}{\big(1-t_1^{-1}\big)\big(1-t_2\big)},
    \hspace{2em}
        \sum_{k_1,k_2 \in \frac{1}{2}\ZZ} b(0,k_1,k_2)
        = - \frac{1+t_1^{-1}t_2}{\big(1-t_1^{-1}\big)\big(1-t_2\big)}.
    \]
\end{proof}

Now applying Lemma \ref{lem: Ell(CC) expansion} and considering the variables
\[
    d_{A} = n, 
    \hspace{1em}
    d_B = m+n+l,
    \hspace{1em}
    d_C = m
\]
with 
$\|\bm{d}\| = \|(d_A,d_B,d_C)\| := 2d_A d_B + 2d_A d_C + 2d_B d_C - d_A^2 - d_B^2 - d_C^2 = 4nm-l^2$
gives
\begin{align*}
    \frac{Z^{\mathbb{DT}}}{\calZ_0^{\mathbb{DT}}}
    =
    ~
    &
    \prod_{n\geq 0}
    \prod_{m>0, l\in \ZZ}
    \prod_{k_1,k_2\in \frac{1}{2}\ZZ }
    \Big(1 - (-Q_{A_0})^n (-Q_B)^{m+n+l} (-Q_{C})^m v^{-k_1} u^{k_2}\Big)^{-b(4nm-l^2,k_1,k_2)}
    \\
    =
    ~
    &
    \prod_{d_C>0}
    \prod_{ d_A, d_B \geq 0}
    \prod_{k_1,k_2\in \frac{1}{2}\ZZ }
    \Big(1 - (-Q_{A_0})^{d_A} (-Q_B)^{d_B} (-Q_{C})^{d_C} v^{-k_1} u^{k_2}\Big)^{-b(\|\bm{d}\|,k_1,k_2)}.
\end{align*}
Note that, a priori we must take the product over all $d_B\in \ZZ$, however as pointed out in \cite[p. 160]{Bryan_DonaldsonThomas} when $d_B<0$ and $d_C>0$ we have $\|\bm{d}\| < -1$ meaning these factors don't contribute. 

Consider the case when $d_C=0$. In this case we have
$
    \|\bm{d}\| = -(d_B - d_A)^2,
$
for which $b(\|\bm{d}\|,k_1,k_2)$ is zero unless $\|\bm{d}\|\in\{-1,0\}$. Now by Corollary \ref{cor: b terms for d_C =0} the terms where $\|\bm{d}\|=0$ are:
\begin{align*}
    &
    M(u^{-1};u,v)^{r}
    \cdot
    M(v^{-1};u,v)^{r}
    \\
    &
    =
    \prod_{ d_A = d_B = 0}
    \multiprod{k_1,k_2\in \frac{1}{2} \ZZ \\ \mbox{s.t. $(k_1,k_2) \neq (0,0)$} }
    \Big(1 - (-Q_{A_0})^{d_A} (-Q_B)^{d_B} (-Q_{C})^{d_C} v^{-k_1} u^{k_2}\Big)^{-b(\|\bm{d}\|,k_1,k_2)}.
\end{align*}

\begin{align*}
    &
    \prod_{N>0}
    \frac{  
        M\big(Q_{A_0}^{N}Q_B^{N} u^{-1}; u,v\big)
        M\big(Q_{A_0}^{N}Q_B^{N} v^{-1}; u,v\big)
    }{
        (1-Q_{A_0}^{N}Q_B^{N})
    }
    \\
    &
    =
    \prod_{ d_A = d_B > 0}
    \prod_{k_1,k_2\in \frac{1}{2}\ZZ}
    \Big(1 - (-Q_{A_0})^{d_A} (-Q_B)^{d_B} (-Q_{C})^{d_C} v^{-k_1} u^{k_2}\Big)^{-b(\|\bm{d}\|,k_1,k_2)}.
\end{align*}
and the terms where $\|\bm{d}\|=-1$ are:
\begin{align*}
    &
    \prod_{N>0}
    \frac{1}{M\big(-Q_{A_0}^{N-1}Q_B^{N}  (uv)^{-\frac{1}{2}}; u,v\big) M\big(-Q_{A_0}^{N}Q_B^{N-1} (uv)^{-\frac{1}{2}}; u,v\big)}
    \\
    &
    =
    \multiprod{ d_A, d_B \geq 0 \\ \mbox{s.t. $d_A = d_B \pm 1$ }}
    \prod_{k_1,k_2\in \frac{1}{2}\ZZ}
    \Big(1 - (-Q_{A_0})^{d_A} (-Q_B)^{d_B} (-Q_{C})^{d_C} v^{-k_1} u^{k_2}\Big)^{-b(\|\bm{d}\|,k_1,k_2)}.
\end{align*}

This completes the proof of Corollary \ref{cor: r=1 product}.

\section*{Acknowledgements}

The author wishes to thank Noah Arbesfeld, Pierre Descombes, Martijn Kool, Nikolas Kuhn and Georg Oberdieck for very useful conversations which contributed to this article. 

%-------------------------------------------------------------------------------------------------------------------------------------------------------------------------------
\bibliography{refs}

\newcommand{\etalchar}[1]{$^{#1}$}
\providecommand{\bysame}{\leavevmode\hbox to3em{\hrulefill}\thinspace}
\providecommand{\MR}{\relax\ifhmode\unskip\space\fi MR }
% \MRhref is called by the amsart/book/proc definition of \MR.
\providecommand{\MRhref}[2]{%
  \href{http://www.ams.org/mathscinet-getitem?mr=#1}{#2}
}
\providecommand{\href}[2]{#2}
\begin{thebibliography}{MNOP06b}

\bibitem[AKL24]{AKL_Vertical}
Noah Arbesfeld, Martijn Kool, and Ties Laarakker, \emph{Vertical vafa-witten invariants and framed sheaves}, in preparation, 2024.

\bibitem[Arb21]{Arbesfeld_K-Theoretic_DT}
Noah Arbesfeld, \emph{K-theoretic {D}onaldson-{T}homas theory and the {H}ilbert scheme of points on a surface}, Algebr. Geom. \textbf{8} (2021), no.~5, 587--625. \MR{4371541}

\bibitem[Arb22]{Arbesfeld_K-Theoretic_Descendent}
\bysame, \emph{K-theoretic descendent series for {H}ilbert schemes of points on surfaces}, SIGMA Symmetry Integrability Geom. Methods Appl. \textbf{18} (2022), Paper No. 078, 16. \MR{4496136}

\bibitem[BB73]{Bialynicki-Birula_Some}
Andrzej Białynicki-Birula, \emph{Some theorems on actions of algebraic groups}, Ann. of Math. (2) \textbf{98} (1973), 480--497. \MR{366940}

\bibitem[BBBBJ15]{BBBJ_Darboux}
Oren Ben-Bassat, Christopher Brav, Vittoria Bussi, and Dominic Joyce, \emph{A `{D}arboux theorem' for shifted symplectic structures on derived {A}rtin stacks, with applications}, Geom. Topol. \textbf{19} (2015), no.~3, 1287--1359. \MR{3352237}

\bibitem[BBD{\etalchar{+}}15]{BBDJS_Symmetries}
Christopher Brav, Vittoria Bussi, Delphine Dupont, Dominic Joyce, and Balázs Szendrői, \emph{Symmetries and stabilization for sheaves of vanishing cycles}, J. Singul. \textbf{11} (2015), 85--151, With an appendix by J\"{o}rg Sch\"{u}rmann. \MR{3353002}

\bibitem[BBJ19]{BBJ_Darboux}
Christopher Brav, Vittoria Bussi, and Dominic Joyce, \emph{A {D}arboux theorem for derived schemes with shifted symplectic structure}, J. Amer. Math. Soc. \textbf{32} (2019), no.~2, 399--443. \MR{3904157}

\bibitem[BBS13]{BBS_Motivic}
Kai Behrend, Jim Bryan, and Balázs Szendrői, \emph{Motivic degree zero {D}onaldson-{T}homas invariants}, Invent. Math. \textbf{192} (2013), no.~1, 111--160. \MR{3032328}

\bibitem[BCY12]{BCY_Orbifold}
Jim Bryan, Charles Cadman, and Ben Young, \emph{The orbifold topological vertex}, Adv. Math. \textbf{229} (2012), no.~1, 531--595. \MR{2854183}

\bibitem[BD21]{BD_Relative}
Christopher Brav and Tobias Dyckerhoff, \emph{Relative {C}alabi-{Y}au structures {II}: shifted {L}agrangians in the moduli of objects}, Selecta Math. (N.S.) \textbf{27} (2021), no.~4, Paper No. 63, 45. \MR{4281260}

\bibitem[Beh09]{Behrend_DonaldsonThomas}
Kai Behrend, \emph{Donaldson-{T}homas type invariants via microlocal geometry}, Ann. of Math. (2) \textbf{170} (2009), no.~3, 1307--1338. \MR{2600874}

\bibitem[BKY18]{BKY_Trace}
Jim Bryan, Martijn Kool, and Benjamin Young, \emph{Trace identities for the topological vertex}, Selecta Math. (N.S.) \textbf{24} (2018), no.~2, 1527--1548. \MR{3782428}

\bibitem[BMP21]{BMP_VafaWitten}
Guillaume Beaujard, Jan Manschot, and Boris Pioline, \emph{Vafa-{W}itten invariants from exceptional collections}, Comm. Math. Phys. \textbf{385} (2021), no.~1, 101--226. \MR{4275783}

\bibitem[BP24]{BP_Geometry}
Jim Bryan and Stephen Pietromonaco, \emph{The geometry and arithmetic of rigid banana nano-manifolds}, in preparation, 2024.

\bibitem[Bra03]{Braden_Hyperbolic}
Tom Braden, \emph{Hyperbolic localization of intersection cohomology}, Transform. Groups \textbf{8} (2003), no.~3, 209--216. \MR{1996415}

\bibitem[Bry21]{Bryan_DonaldsonThomas}
Jim Bryan, \emph{The {D}onaldson-{T}homas partition function of the banana manifold}, Algebr. Geom. \textbf{8} (2021), no.~2, 133--170, With an appendix coauthored with Stephen Pietromonaco. \MR{4174287}

\bibitem[CKL17]{CKL_Torus}
Huai-Liang Chang, Young-Hoon Kiem, and Jun Li, \emph{Torus localization and wall crossing for cosection localized virtual cycles}, Adv. Math. \textbf{308} (2017), 964--986. \MR{3600080}

\bibitem[Dav17]{Davison_Critical}
Ben Davison, \emph{The critical {C}o{HA} of a quiver with potential}, Q. J. Math. \textbf{68} (2017), no.~2, 635--703. \MR{3667216}

\bibitem[Dav19]{Davison_Refined}
\bysame, \emph{Refined invariants of finite-dimensional {J}acobi algebras}, Algebr. Geom. (2019), to appear.

\bibitem[Des22a]{Descombes_Cohomological}
Pierre Descombes, \emph{Cohomological {DT} invariants from localization}, J. Lond. Math. Soc. (2) \textbf{106} (2022), no.~4, 2959--3007. \MR{4524190}

\bibitem[Des22b]{Descombes_Hyperbolic}
\bysame, \emph{Hyperbolic localization of the donaldson-thomas sheaf}, arXiv:2201.12215, 2022.

\bibitem[DM15]{DM_Motivic_Loop}
Ben Davison and Sven Meinhardt, \emph{Motivic {D}onaldson-{T}homas invariants for the one-loop quiver with potential}, Geom. Topol. \textbf{19} (2015), no.~5, 2535--2555. \MR{3416109}

\bibitem[DM17]{DM_Motivic_-2}
\bysame, \emph{The motivic {D}onaldson-{T}homas invariants of {$(-2)$}-curves}, Algebra Number Theory \textbf{11} (2017), no.~6, 1243--1286. \MR{3687097}

\bibitem[DM20]{DM_Cohomological}
\bysame, \emph{Cohomological {D}onaldson-{T}homas theory of a quiver with potential and quantum enveloping algebras}, Invent. Math. \textbf{221} (2020), no.~3, 777--871. \MR{4132957}

\bibitem[Dri13]{Drinfeld_Algebraic}
Vladimir Drinfeld, \emph{On algebraic spaces with an action of {$\mathbb{G}_m$}}, arXiv:1308.2604, 2013.

\bibitem[DS09]{DS_Milnor}
Alexandru Dimca and Balázs Szendrői, \emph{The {M}ilnor fibre of the {P}faffian and the {H}ilbert scheme of four points on {$\Bbb C^3$}}, Math. Res. Lett. \textbf{16} (2009), no.~6, 1037--1055. \MR{2576692}

\bibitem[GK19]{GK_Rank2}
Lothar Göttsche and Martijn Kool, \emph{A rank 2 {D}ijkgraaf-{M}oore-{V}erlinde-{V}erlinde formula}, Commun. Number Theory Phys. \textbf{13} (2019), no.~1, 165--201. \MR{3951108}

\bibitem[GP99]{GP_Localization}
Tom Graber and Rahul Pandharipande, \emph{Localization of virtual classes}, Invent. Math. \textbf{135} (1999), no.~2, 487--518. \MR{1666787}

\bibitem[IKV09]{IKV_Refined}
Amer Iqbal, Can Kozçaz, and Cumrun Vafa, \emph{The refined topological vertex}, J. High Energy Phys. (2009), no.~10, 069, 58. \MR{2607441}

\bibitem[Joy15]{Joyce_Classical}
Dominic Joyce, \emph{A classical model for derived critical loci}, J. Differential Geom. \textbf{101} (2015), no.~2, 289--367. \MR{3399099}

\bibitem[KL19]{KL_Local}
Atsushi Kanazawa and Siu-Cheong Lau, \emph{Local {C}alabi-{Y}au manifolds of type {$\tilde A$} via {SYZ} mirror symmetry}, J. Geom. Phys. \textbf{139} (2019), 103--138. \MR{3913509}

\bibitem[KLT22]{KLT_blowup}
Nikolas Kuhn, Oliver Leigh, and Yuuji Tanaka, \emph{The blowup formula for the instanton part of vafa-witten invariants on projective surfaces}, arXiv:2205.12953, 2022.

\bibitem[KS08]{KS_Stability}
Maxim Kontsevich and Yan Soibelman, \emph{Stability structures, motivic donaldson-thomas invariants and cluster transformations}, arXiv:0811.2435, 2008.

\bibitem[KS11]{KS_Cohomological}
\bysame, \emph{Cohomological {H}all algebra, exponential {H}odge structures and motivic {D}onaldson-{T}homas invariants}, Commun. Number Theory Phys. \textbf{5} (2011), no.~2, 231--352. \MR{2851153}

\bibitem[Lei20]{Leigh_Unweighted}
Oliver Leigh, \emph{Unweighted {D}onaldson-{T}homas theory of the banana 3-fold with section classes}, Q. J. Math. \textbf{71} (2020), no.~3, 867--942. \MR{4142715}

\bibitem[MMNS12]{MMNS_Motivic}
Andrew Morrison, Sergey Mozgovoy, Kentaro Nagao, and Balázs Szendrői, \emph{Motivic {D}onaldson-{T}homas invariants of the conifold and the refined topological vertex}, Adv. Math. \textbf{230} (2012), no.~4-6, 2065--2093. \MR{2927365}

\bibitem[MN15]{MN_Motivic}
Andrew Morrison and Kentaro Nagao, \emph{Motivic {D}onaldson-{T}homas invariants of small crepant resolutions}, Algebra Number Theory \textbf{9} (2015), no.~4, 767--813. \MR{3352820}

\bibitem[MNOP06a]{MNOP_I}
Davesh Maulik, Nikita Nekrasov, Andrei Okounkov, and Rahul Pandharipande, \emph{Gromov-{W}itten theory and {D}onaldson-{T}homas theory. {I}}, Compos. Math. \textbf{142} (2006), no.~5, 1263--1285. \MR{2264664}

\bibitem[MNOP06b]{MNOP_II}
\bysame, \emph{Gromov-{W}itten theory and {D}onaldson-{T}homas theory. {II}}, Compos. Math. \textbf{142} (2006), no.~5, 1286--1304. \MR{2264665}

\bibitem[Mor21]{Morishige_Genus_Multi}
Nina Morishige, \emph{Genus zero gopakumar-vafa invariants of multi-banana configurations}, arXiv:2102.07965, 2021.

\bibitem[Mor22]{Morishige_Genus}
\bysame, \emph{Genus 0 {G}opakumar-{V}afa invariants of the banana manifold}, Q. J. Math. \textbf{73} (2022), no.~1, 175--212. \MR{4395077}

\bibitem[MP20]{MP_Attractor}
Sergey Mozgovoy and Boris Pioline, \emph{Attractor invariants, brane tilings and crystals}, arXiv:2012.14358, 2020.

\bibitem[NO16]{NO_Membranes}
Nikita Nekrasov and Andrei Okounkov, \emph{Membranes and sheaves}, Algebr. Geom. \textbf{3} (2016), no.~3, 320--369. \MR{3504535}

\bibitem[NY05]{NY_Instanton_I}
Hiraku Nakajima and Kōta Yoshioka, \emph{Instanton counting on blowup. {I}. 4-dimensional pure gauge theory}, Invent. Math. \textbf{162} (2005), no.~2, 313--355. \MR{2199008}

\bibitem[Oko01]{Okounkov_Infinite}
Andrei Okounkov, \emph{Infinite wedge and random partitions}, Selecta Math. (N.S.) \textbf{7} (2001), no.~1, 57--81. \MR{1856553}

\bibitem[OR07]{OR_Random}
Andrei Okounkov and Nicolai Reshetikhin, \emph{Random skew plane partitions and the {P}earcey process}, Comm. Math. Phys. \textbf{269} (2007), no.~3, 571--609. \MR{2276355}

\bibitem[ORV06]{ORV_Quantum}
Andrei Okounkov, Nikolai Reshetikhin, and Cumrun Vafa, \emph{Quantum {C}alabi-{Y}au and classical crystals}, The unity of mathematics, Progr. Math., vol. 244, Birkh\"{a}user Boston, Boston, MA, 2006, pp.~597--618. \MR{2181817}

\bibitem[Ric19]{Richarz_Spaces}
Timo Richarz, \emph{Spaces with {$\mathbb{G}_m$}-action, hyperbolic localization and nearby cycles}, J. Algebraic Geom. \textbf{28} (2019), no.~2, 251--289. \MR{3912059}

\bibitem[Rud17]{Ruddat_Perverse}
Helge Ruddat, \emph{Perverse curves and mirror symmetry}, J. Algebraic Geom. \textbf{26} (2017), no.~1, 17--42. \MR{3570582}

\bibitem[Sze16]{Szendroi_Cohomological}
Balázs Szendrői, \emph{Cohomological {D}onaldson-{T}homas theory}, String-{M}ath 2014, Proc. Sympos. Pure Math., vol.~93, Amer. Math. Soc., Providence, RI, 2016, pp.~363--396. \MR{3526001}

\bibitem[Tho00]{Thomas_Holomorphic}
Richard~P. Thomas, \emph{A holomorphic {C}asson invariant for {C}alabi-{Y}au 3-folds, and bundles on {$K3$} fibrations}, J. Differential Geom. \textbf{54} (2000), no.~2, 367--438. \MR{1818182}

\bibitem[Wae08]{Waelder_Equivariant}
Robert Waelder, \emph{Equivariant elliptic genera and local {M}c{K}ay correspondences}, Asian J. Math. \textbf{12} (2008), no.~2, 251--284. \MR{2439263}

\bibitem[You10]{Young_Generating}
Benjamin Young, \emph{Generating functions for colored 3{D} {Y}oung diagrams and the {D}onaldson-{T}homas invariants of orbifolds}, Duke Math. J. \textbf{152} (2010), no.~1, 115--153, With an appendix by Jim Bryan. \MR{2643058}

\end{thebibliography}
\bibliographystyle{amsalpha}
%-------------------------------------------------------------------------------------------------------------------------------------------------------------------------------
\end{document}